\documentclass[12pt]{article}
\usepackage{amsfonts}
\usepackage{amsmath}
\usepackage{amssymb}
\usepackage[mathscr]{eucal}
\usepackage{fullpage}
\usepackage{times}
\usepackage[usenames]{color}
\usepackage{hyperref}    

\def\eod{\vrule height 6pt width 5pt depth 0pt}
\newenvironment{proof}{\noindent {\bf Proof:} \hspace{.2em}}
                      {\hspace*{\fill}{\eod}}

\newcommand{\floor}[1]{\left\lfloor #1 \right\rfloor}

\newcommand{\exc}{\mathsf{exc}}

\newcommand{\des}{\mathsf{des}}
\newcommand{\ides}{\mathsf{ides}}
\newcommand{\lngt}{\mathsf{len}}
\newcommand{\asc}{\mathsf{asc}}
\newcommand{\DescSet}{\mathsf{Des\_Set}}
\newcommand{\inv}{\mathsf{inv}}

\newcommand{\np}{ \mathsf{pos\_n}}
\newcommand\twoonethree{\operatorname{2-13}}
\newcommand\threeonetwo{\operatorname{31-2}}

\newcommand{\SSS}{\mathfrak{S}}
\newcommand{\BB}{\mathfrak{B}}
\newcommand{\SB}{\mathrm{SgnB}}
\newcommand{\ATS}{\mathrm{TSA}}
\newcommand{\TSB}{\mathrm{TSB}}
\newcommand{\TSD}{\mathrm{TSD}}

\newcommand{\Negs}{\mathsf{Negs}}

\newcommand{\lpk}{\mathsf{lpk}}
\newcommand{\DD}{\mathfrak{D}}
\newcommand{\nhalf}{\lfloor n/2 \rfloor}
\newcommand{\nmhalf}{\lfloor (n-1)/2 \rfloor}
\newcommand{\nmrhalf}{\lfloor (n-r)/2 \rfloor}

\newcommand{\AAA}{\mathcal{A}}

\newcommand{\cnos}{\mathrm{ \operatorname{c-o-s}}}
\newcommand{\invpair}{ \mathsf{invpair}}

\newcommand{\PP}{ \mathrm{PalindPoly}}

\newcommand{\pos}{ \mathsf{pos}}

\newtheorem{theorem}{Theorem}

\newtheorem{corollary}[theorem]{Corollary}
\newtheorem{conjecture}[theorem]{Conjecture}
\newtheorem{question}[theorem]{Question}
\newtheorem{remark}[theorem]{Remark}

\newtheorem{lemma}[theorem]{Lemma}

\newcommand{\ZZ}{ \mathbb{Z}}
\newcommand{\QQ}{ \mathbb{Q}}

\newcommand{\comment}[1]{}
\newcommand{\red}[1]{\textcolor{red}{#1}}
\newcommand{\magenta}[1]{}
\newcommand{\blue}[1]{\textcolor{blue}{#1}}
\newcommand{\ob}[1]{\overline{#1}}

\begin{document}

\title{Gamma positivity of the Descent based Eulerian
polynomial in positive elements of Classical Weyl Groups}

\author{Hiranya Kishore Dey\\ 
Department of Mathematics\\
Indian Institute of Technology, Bombay\\
Mumbai 400 076, India.\\
email: hkdey@math.iitb.ac.in
\and
Sivaramakrishnan Sivasubramanian\\
Department of Mathematics\\
Indian Institute of Technology, Bombay\\
Mumbai 400 076, India.\\
email: krishnan@math.iitb.ac.in 
}

\maketitle

\section{Introduction}
\label{chap:intro}
Let $f(t) \in \QQ[t]$ be a degree $n$ univariate polynomial with 
$f(t) = \sum_{i=0}^n a_i t^i$ where $a_i \in \QQ$ with 
$a_n \not= 0$.  Let $r$ be the least non-negative integer such that 
$a_r \not= 0$.  Define $\lngt(f) = n-r$.  $f(t)$ is 
said to be palindromic if $a_{r+i} = a_{n-i}$ for $0 \leq i \leq \floor{(n-r)/2}$.  
Define the {\it center of 
symmetry} of $f(t)$ to be $\cnos(f(t)) =  (n+r)/2$.  
Note that for a palindromic 
polynomial $f(t)$, its center of symmetry $\cnos(f(t))$ could be half integral.

Let $\PP_{(n+r)/2, r}(t)$ denote the vector space of palindromic 
univariate polynomials $f(t)$ with least non zero exponent of $t$ being
at least $r$ and with $\cnos(f(t)) = (n+r)/2$,.  Clearly, 
$\Gamma = \{ t^{r+i}(1+t)^{n-r-2i}: 0 \leq i \leq \nmrhalf \}$
is a basis of $\PP_{(n+r)/2, r}(t)$.  
Hence, if $f(t) \in \PP_{(n+r)/2, r}(t)$,
then we can write
$f(t) = \sum_{i=0}^{\floor{(n-r)/2}} \gamma_{n,i} t^{r+i} (1+t)^{n-r-2i}$.
$f(t)$ is said to be {\sl gamma positive} if $\gamma_{n,i} \geq 0$ for 
all $i$ (that is, if $f(t)$ has positive coefficients when expressed in 
the basis $\Gamma$).

For a positive integer $n$, let $[n] = \{1,2,\ldots,n\}$.  Denote 
by $\SSS_n$ the symmetric group and let $\AAA_n \subseteq \SSS_n$ 
be the alternating group on $[n]$.  
For $\pi \in \SSS_n$, let $\DescSet(\pi) = \{i \in [n-1]: \pi_i > \pi_{i+1} \}$ 
be its set of descents and  
let $\des(\pi) = |\DescSet(\pi)|$ be its number of descents.  Let
$\asc(\pi) = n  -1 - \des(\pi)$ denote its number of ascents.
Further, let 

\vspace{-5 mm}

\begin{eqnarray}
\label{eqn:Sn} 
A_n(t) & = & \sum_{\pi \in \SSS_n} t^{\des(\pi)} 
\mbox{ 
\hspace{3 mm} 
and 
\hspace{3 mm}
} 
A_n(s,t)  =  \sum_{\pi \in \SSS_n} t^{\des(\pi)} s^{\asc(\pi)} 
= \sum_{i=0}^{n-1} a_{n,i} s^{n-1-i}t^i
\\  
\label{eqn:An} 
A_n^+(t) & = & \sum_{\pi \in \AAA_n} t^{\des(\pi)} 
\mbox{ 
\hspace{3 mm}
and 
\hspace{3 mm}
}  
A_n^+(s,t)  =  \sum_{\pi \in \AAA_n} t^{\des(\pi)} s^{\asc(\pi)}
= \sum_{i=0}^{n-1} a_{n,i}^+ s^{n-1-i}t^i \\
\label{eqn:SnAn} 
A_n^-(t) & = & \sum_{\pi \in (\SSS_n - \AAA_n)} t^{\des(\pi)}  
\mbox{ 
\hspace{1 mm}
and 
\hspace{1 mm}
}  
A_n^-(s,t)   =  \sum_{\pi \in (\SSS_n - \AAA_n)} t^{\des(\pi)} s^{\asc(\pi)}
= \sum_{i=0}^{n-1} a_{n,i}^- s^{n-1-i}t^i 
\end{eqnarray}

It is well known that the Eulerian polynomials $A_n(t)$ are  
palindromic (see Graham, Knuth and Patashnik \cite{graham-knuth-patashnik}).
Gamma positivity of $A_n(t)$  was first proved by 
Foata and Sch{\"u}tzenberger (see \cite{foata-schutzenberger-eulerian} and 
see Theorem  \ref{thm:foata_schut_Sn} as well).
Foata and Strehl \cite{foata-strehl} later used a group action
based proof which has been termed as ``valley hopping" by
Shapiro, Woan and Getu \cite{shapiro-woan-getu_runs_slides_moments}.
The same proof shows that the Narayana polynomials 
are gamma positive (see Petersen's book 
\cite[Chapter 4]{petersen-eulerian-nos-book}).  

It is a well known result of MacMahon \cite{macmahon-book} that both 
descents and excedances are equidistributed over $\SSS_n$. 
Several refinements of this result are known both when enumeration is done 
with respect to descents $\des()$ and with respect to excedance $\exc()$.  
For two statistics $\twoonethree, 
\threeonetwo: \SSS_n \mapsto \ZZ_{\geq 0}$,
Br{\"a}nd{\'e}n \cite{branden-actions_on_perms_unimodality_descents}
and Shin-and-Zeng 
\cite{shin-zeng-eulerian-continued-fraction,shin-zeng_symmetric_unimodal_excedances_colored_perms} 
have shown a $p,q$-refinement by showing  that 
the polynomial

\vspace{-6 mm}

$$A_n(p,q,t) = \sum_{\pi \in \SSS_n} p^{\twoonethree(\pi)}q^{\threeonetwo(\pi)}t^\des(\pi) 
= \sum_{i=0}^{\nhalf} a_{n,i}(p,q) t^i(1+t)^{n-2i}$$
where the $a_{n,i}(p,q)$'s are polynomials with positive coefficients.

The survey \cite{athanasiadis-survey-gamma-positivity} by Athanasiadis 
and the Handbook chapter \cite{branden-unimodality_log_concavity} by 
Br{\"a}nd{\'e}n contain a wealth of information related to gamma positivity
of polynomials arising in combinatorics.  
The polynomials $A_n^+(t)$ and $A_n^-(t)$ have been studied by Tanimoto 
\cite{tanimoto-signed} who gave recurrences for $a_{n,k}^+, a_{n,k}^-$, the 
coefficients of $t^k$ in these polynomials.  To the best of our knowledge, 
there are no results on gamma positivity of the polynomials $A_n^+(t)$.  
In this work, we show the following results.

\begin{theorem}
\label{thm:An_01_mod4}
For $n\geq 1$, both $A_n^+(s,t)$ and $A_n^-(s,t)$ are 
gamma positive iff $n \equiv 0,1$ (mod 4).  Further, both the polynomials have
the same center of symmetry.
\end{theorem}

\begin{theorem}
  \label{thm:An_more_gamma_positive}
When $n \equiv 2$ (mod 4), the polynomials $A_n^+(t)$ and $A_n^-(t)$ 
can be written as a sum of two gamma positive polynomials.
\end{theorem}

\begin{theorem}
\label{thm:An_still_more_gamma_positive}
Let $n \equiv 3$ (mod 4).  Then,
$A_{4m+3}^+(t)$ and $A_{4m+3}^-(t)$ can be written as a sum of 
three gamma positive polynomials.
\end{theorem}

Theorem \ref{thm:An_01_mod4} thus gives a different refinement of 
Foata and Sch{\"u}tzenberger's Gamma positivity result 
(see Theorem \ref{thm:foata_schut_Sn}) when $n \equiv 0,1$ (mod 4).
Theorems \ref{thm:An_more_gamma_positive} and 
 \ref{thm:An_still_more_gamma_positive} refine 
the univariate version of Theorem \ref{thm:foata_schut_Sn}.

We generalize our results to the case when descents are summed up over
the elements with positive sign in classical Weyl groups.  
Let $\BB_n$ denote the group of signed permutations on 
$T_n = \{-n, -(n-1), \ldots,-1,1,2,\ldots,n \}$, that is $\sigma \in \BB_n$ 
consists of all permutations of $T_n$ that satisfy $\sigma(-i) = -\sigma(i)$
for all $i \in [n]$.  Similarly, let $\DD_n\subseteq \BB_n$ denote the
subset consisting of those elements of $\BB_n$ which have an even
number of negative entries.  As both $\BB_n$ and $\DD_n$ are Coxeter groups,
they have a natural notion of descent associated to them.  Similar to 
the classical Eulerian polynomials, we thus have the Eulerian polynomials of 
type-B and type-D. 
Gamma positivity of the type-B Eulerian polynomial was 
shown by Chow \cite{chow-certain_combin_expansions_eulerian} and 
Petersen \cite{petersen-enriched_p_partitions_peak_algebras}.
Gamma positivity of the type-D Eulerian polynomial was 
shown by Stembridge \cite{stembridge-coxeter-cones} and by Chow
\cite{chow-certain_combin_expansions_eulerian}.  There is a natural
notion of length in these groups and we get results when we 
enumerate descents in these groups but restrict the sum over 
elements with even length.  
Our results for Type-B Weyl groups are  Theorems \ref{thm:type_B_recurrence} and 
\ref{thm:main_result_for_Bn}.  Similarly for Type-D Weyl groups, our main result is Theorem 
\ref{thm:main_result_for_Dn}.

For $\pi \in \SSS_n$, let $\exc(\pi) = |\{i \in [n-1]: \pi_i > i  \}|$ 
denote its number of excedances.  Define 

\vspace{-2 mm}

\begin{equation}
\label{eqn:Sn_exc} 
A_{n,\exc}(t)  =  \sum_{\pi \in \SSS_n} t^{\exc(\pi)} 
\mbox{ 
\hspace{3 mm} 
and 
\hspace{3 mm}
} 
A_{n,\exc}^+(t)  =  \sum_{\pi \in \AAA_n} t^{\exc(\pi)} 
\end{equation}

As mentioned earlier, it is well known that both 
descents and excedances are equidistributed over $\SSS_n$. 
It is also easy to see that they are not equidistributed over $\AAA_n$.
That is, for all $n$, we have $A_n(t) = A_{n,\exc}(t)$ while
$A_n^+(t) \not= A_{n,\exc}^+(t)$.   In this paper we deal with
gamma positivity of the descent based polynomials $A_n^+(t)$.  
Similar results can be proved about
the polynomials $A_{n,\exc}^+(t)$.   We plan to do that in a 
subsequent paper.

\section{Preliminaries on Gamma positive polynomials}
\label{subsec:gamma_prelims}

Most of the univariate polynomials in this work have homogeneous 
bivariate counterparts with the following slightly more general
definition of gamma positivity. 
Let $f(s,t) = \sum_{i=0}^n a_i s^{n-i}t^i$ be 
a homogeneous bivariate polynomial with $a_n \not= 0$ and let $r$ be the smallest
non-negative integer such that $a_r \not= 0$.  As before, define
$f(s,t)$ to be palindromic if $a_{r+i} = a_{n-i}$ for all $i$ and 
define $\cnos(f(s,t)) = (n+r)/2$.  

Let $f(s,t)$ have $\cnos(f(s,t))=n/2$.  $f(s,t)$ is said to be bivariate 
gamma positive if 
$f(s,t) = \sum_{i=0}^{n/2} \gamma_{n,i} (st)^{r+i}(s+t)^{n-r-2i}$ with $\gamma_{n,i}
\geq 0$ for all $i$.  Clearly, if $f(s,t)$ is 
a bivariate gamma positive polynomial, then $f(t) = f(s,t)|_{s=1}$ is 
clearly a univariate gamma positive palindromic polynomial with the same
center of symmetry.  We need the following lemmas.  Some of them are 
elementary and so we omit proofs.

\begin{lemma}
\label{lem:prod_bivariate_gamma_nonneg}
Let $f_1(s,t)$ and $f_2(s,t)$ be two bivariate gamma positive 
polynomials with respective centers of symmetry 
$\cnos(f_1(s,t)) = r_1 + (n-r_1)/2$ and 
$\cnos(f_2(s,t)) = r_2 + (n-r_2)/2$.
Then $f_1(s,t) f_2(s,t)$ is gamma positive with 
$\cnos(f_1(s,t)f_2(s,t)) =  (r_1+r_2) + (2n-r_1-r_2)/2$.
\end{lemma}

Let $D$ be the operator $\displaystyle \left( \frac{d}{ds} + \frac{d}{dt}
\right)$ in $\QQ[s,t]$.  

\begin{lemma}
\label{lem:derivative_gamma_nonneg}
Let $f(s,t)$ be a bivariate gamma positive polynomial with 
$\cnos(f(s,t)) = n/2$.  Then, $g(s,t) = D f(s,t)$
is gamma positive with $\cnos(g(s,t)) = (n-1)/2$.
\end{lemma}
\begin{proof}
Let,
\begin{eqnarray*}
f(s,t) & = & \sum_{i=0}^{\nhalf} \gamma_{n,i} (st)^i(s+t)^{n-2i} \\
D f(s,t) & = & \sum_{i=0}^{\nhalf} i \gamma_{n,i} (st)^{i-1} (s+t)^{n-2i+1} +
\sum_{i=0}^{\nhalf} 2 (n-2i) \gamma_{n,i} (st)^i (s+t)^{n-2i-1} 
\end{eqnarray*}
Hence $\cnos(Df(s,t)) = (n-1)/2$ completing the proof.
\end{proof}

\vspace{3 mm}

The following are easy corollaries.

\begin{corollary}
\label{cor:gamma_nonneg}
Let $f(s,t)$ be a bivariate gamma positive polynomial with 
$\cnos(f(s,t)) = n/2$.  
\begin{enumerate}
\item Then, for a natural number $\ell$, 
$\displaystyle g(s,t) = D^{\ell} f(s,t)$ is 
gamma positive with $\cnos(g(s,t)) = (n-\ell)/2$.
\item Then, $h(s,t) = (st)^i f(s,t)$ is gamma positive with 
$\cnos( h(s,t) ) = i+ n/2$
\item Then, $h(s,t) = (s+t) f(s,t)$ is gamma positive with
$\cnos( h(s,t) ) = (n+1)/2$
\end{enumerate}
\end{corollary}

\begin{lemma}
  \label{lem:more_gamma_positive}
Let $f(t)$ be gamma positive with $\cnos(f(t)) = n/2$ and with
$\lngt(f(t))$ being odd.  Then, $f(t)$ is the sum of 
two gamma positive sequences $p_1(t)$ and $p_2(t)$ with centers 
of symmetry $(n-1)/2$ and $(n+1)/2$ respectively.  Further,
both $p_1(t)$ and $p_2(t)$ have even length.
\end{lemma}
\begin{proof}
Let the degree of $f(t)$ be 
$d$ and let $f(t) = \sum_{i=n-d}^{\nhalf} \gamma_i t^i (1+t)^{n-2i}$.  We are given
$\cnos(f(t))= k=n/2$.  Thus,
\begin{eqnarray*}
f(t)  =  \sum_i^{\nhalf} \gamma_i t^i (1+t)^{n-2i-1}(1+t)  
  & = &  \left( \sum_i^{\nhalf} \gamma_i t^i(1+t)^{n-1-2i} \right) + 
 \left( \sum_i^{\nhalf} \gamma_i t^{i+1} (1+t)^{n-1-2i} \right) \\
& = & p_1(t) + p_2(t) 
\end{eqnarray*}

It is easy to see that $p_1(t)$ and $p_2(t)$ are gamma positive 
with respective centers of symmetry $(n-1)/2$ and $(n+1)/2$. 
Note that oddness of length is required in the proof and 
that $p_1(t)$ and $p_2(t)$ have even length.
\end{proof}

\subsection{Gamma positivity of the Eulerian numbers}
\label{subsec:foata_schut_proof}
We sketch a proof that the bivariate Eulerian polynomials $A_n(s,t)$ 
are gamma positive.  We do this as it sets up the stage for 
proofs of other  results in this paper.
The result is due to 
Foata and Sch{\"u}tzenberger (see \cite{foata-schutzenberger-eulerian})
and the proof given below is a modification of the 
proof given in \cite{visontai-remarks-ejc}.  

\begin{theorem}[Foata and Sch{\"u}tzenberger]
  \label{thm:foata_schut_Sn}
With $A_n(s,t)$ as defined in \eqref{eqn:Sn}, we have 
\begin{equation}
  A_n(s,t) = \sum_{i=0}^{\nmhalf} \gamma_{n,i} (st)^i (s+t)^{n-1-2i}
\end{equation}
where the $\gamma_{n,i}$ are positive integers for all 
$n,i$.
\end{theorem}
\begin{proof}(Sketch)
By induction on $n$.  The base case when $n=2$ is easy to see.
Foata and Sch{\"u}tzenberger showed the following recurrence 

\begin{equation}
  \label{eqn:rec_biv_eulerian}
  A_{n+1}(s,t) = (s+t)A_n(s,t) + st D A_n(s,t).
\end{equation}

This can be seen by adding the symbol $(n+1)$ in the $n+1$
places of each permutation $\pi \in \SSS_n$.  The first term
$(s+t)A_n(s,t)$ accounts for those permutations in $\SSS_{n+1}$
in which the letter $(n+1)$ appears in the first or the last 
position.
The term $\displaystyle st D A_n(s,t)$ is the contribution
of all $\pi \in \SSS_{n+1}$ in which the letter $(n+1)$ appears
in positions $r$ for $2 \leq r \leq n$.

\vspace{2 mm}

Let $\displaystyle T = (s+t) + st D$ be 
an operator in $\QQ[s,t]$.  By induction, $A_n(s,t)$ is gamma
positive with $\cnos(A_n(s,t)) = (n-1)/2$.  By Corollary 
\ref{cor:gamma_nonneg},  both $(s+t)A_n(s,t)$ and $stDA_n(s,t)$
are gamma positive with centers of symmetry $n/2$.  Thus
their sum is also gamma positive with the same 
center of symmetry $n/2$.
From the above argument, the following recurrence is easy to obtain

\vspace{-3 mm}

\begin{equation}
  \label{eqn:gamma_Sn}
\gamma_{n+1,i} = (i+1)\gamma_{n,i} + 2(n+1-2i)\gamma_{n,i-1}.
\end{equation}

Note that the above recurrence together with the initial conditions 
$\gamma_{1,0} = 1$ or $\gamma_{2,0} = 1, \gamma_{2,1} = 0$ settles
non-negativity of $\gamma_{n,i}$ for all $n,i$, completing the 
proof.
\end{proof}

\begin{corollary}
\label{cor:gamma_coeff_oddness}
From \eqref{eqn:gamma_Sn}, we get that $\gamma_{n,i}$ is even for all
$n \geq 1$ and for all $i \geq 1$ while $\gamma_{n,0} = 1$ for all
$n$.  Thus, 
for all $n$, the only odd gamma coefficient is $\gamma_{n,0}$.
\end{corollary}

\section{Recurrences for $A_n^+(s,t)$ and $A_n^-(s,t)$}

Recall $A_n^+(s,t)$ and $A_n^-(s,t)$ were defined in 
\eqref{eqn:An} and \eqref{eqn:SnAn}.
We first consider  palindromicity of $A_n^+(s,t)$ and $A_n^-(s,t)$.

\begin{lemma}
\label{lem:01_mod_4}
For $n\geq 1$, the polynomials
$A_n^+(s,t)$ and $A_n^-(s,t)$ are palindromic iff 
$n \equiv 0,1$ mod 4.
\end{lemma}
\begin{proof}
Let $\pi = x_1,x_2,\ldots,x_n \in \SSS_n$.  Define $f: \SSS_n \mapsto \SSS_n$ 
by $f(\pi) = n+1-x_1,n+1-x_2, \ldots, n+1-x_n$.  Clearly, $f$ is a bijection and
$\inv_A(\pi) + \inv_A(f(\pi)) = \binom{n}{2}$.  Hence, if $n \equiv 0,1$ mod 4
and $\pi \in \AAA_n$, then $f(\pi) \in \AAA_n$.  

If $n \equiv 2,3$ mod 4, then the map $f$ flips the parity of $\inv(\pi)$ and 
hence $a_{n,k}^+ = a_{n,n-1-k}^-$.  Further, $a_{n,0}^+ = 1$ and 
$a_{n,n-1}^+ = a_{n,0}^- = 0$ and hence $A_n^+(s,t)$ and $A_n^-(s,t)$ are 
not palindromic.
\end{proof}

For $\pi \in \AAA_n$ let $\pi' = \pi$ restricted to $[n-1]$.  
It is easy to see that both $\pi' \in \AAA_{n-1}$ and $\pi' \in 
\SSS_{n-1} - \AAA_{n-1}$ are possible.  A similar statement is true 
when we restrict $\pi \in \SSS_n - \AAA_n$ to $[n-1]$.  

In subsequent results, we will need
to keep track of the position of the letter $n$ in $\pi$.  
Let $\pi= x_1,x_2, \ldots, x_{n} \in \SSS_n$.
We term the left-most position before $x_1$ as the initial or `zero'-th gap
and the right-most position after $x_{n}$ as the final 
gap.
For $1 \leq i \leq n$, we denote the gap between $x_i$ and 
$x_{i+1}$ as the $`i$'-th gap of a permutation.
Define $\np(\pi)$ to be the index $i$ such that
$\pi(i)=n$ and recall $\pi'$ is  $\pi$ restricted to $[n-1]$.
We start with the following definitions. 

\begin{enumerate}
\item  $\AAA_n^{i}$= $\{ \pi \in {\AAA_n}: \np(\pi)=i+1\} $.  That is, 
$\AAA_n^i$ consists of  those $\pi \in {\AAA_{n}}$ which arise 
from all possible $\pi' \in \SSS_{n-1}$ 
by inserting the letter $``n"$ in the $i$-th gap for $0 \leq i \leq n-1 $.  
\item $[\AAA_{n-1} \rightarrow \AAA_n^i]= \{ \pi \in \AAA_n :\np(\pi)=i+1,  
\pi' \in \AAA_{n-1} \} $.  That is, 
$[\AAA_{n-1} \rightarrow \AAA_n^i]$ consists of those $\pi \in 
\AAA_n$ which arise from $\pi' \in \AAA_{n-1}$ by inserting the letter $``n"$ in 
the $i$-th gap for $0 \leq i \leq n-1 $.

\item $[(\SSS_{n-1}-\AAA_{n-1}) \rightarrow \AAA_n^i]=\{ 
\pi \in {\AAA_{n}} :\np(\pi)=i+1,  \pi' \in (\SSS_{n-1}-\AAA_{n-1}) \}$.
We think of $[(\SSS_{n-1}-\AAA_{n-1}) \rightarrow \AAA_n^i]$ 
as those $\pi \in \AAA_n$ which arise from 
$\pi' \in (\SSS_{n-1}- \AAA_{n-1})$ 
by inserting the letter $``n"$ in the $i$-th gap for $0 \leq i \leq n-1 $.
\item $(\SSS_n-\AAA_n)^i$= $\{ \pi \in \SSS_n-\AAA_n :\np(\pi)=i+1,  
\pi' \in \SSS_{n-1} \} $. 
We think of $(\SSS_n-\AAA_n)^i$ as those $\pi \in \SSS_n-\AAA_n$ which
arise from all $\pi' \in \SSS_{n-1}$ by inserting 
the letter $`n'$ in the $i$-th position for 
$0 \leq i \leq n-1 $.
\item $[\AAA_{n-1} \rightarrow  (\SSS_n-\AAA_n)^{i}]= \{ \pi 
\in \SSS_n-\AAA_n:\np(\pi)=i+1,  \pi' \in \AAA_{n-1} \}$ . 
We think of $[\AAA_{n-1} \rightarrow  (\SSS_n-\AAA_n)^{i}]$ as 
those $\pi \in \SSS_n-\AAA_n$ which arise from all 
$\pi' \in  \AAA_{n-1}$  by inserting the letter $``n"$ in the 
$i$-th position for $0 \leq i \leq n-1 $.
\item $[(\SSS_{n-1}-\AAA_{n-1})\rightarrow (\SSS_n-\AAA_n)^i]=
\{\pi \in \SSS_n-\AAA_n:\np(\pi)=i+1, \pi'\in \SSS_{n-1}-\AAA_{n-1} \}.$ 
We can think of $[(\SSS_{n-1}-\AAA_{n-1})
\rightarrow  (\SSS_n-\AAA_n)^i]$
as those $\pi \in \SSS_n-\AAA_n$ 
which arise from all $\pi' \in  \SSS_{n-1}- \AAA_{n-1}$ by inserting 
the letter $`n'$ in the $i$-th position for $0 \leq i \leq n-1 $.
\end{enumerate} 

The way permutations in $\AAA_n$ 
arise from elements in $\AAA_{n-1}$ and $\SSS_{n-1} - \AAA_{n-1}$ depends on
the parity of $n$ and is given by the following two lemmas.  
Since the proofs are simple, we omit them.

\begin{lemma}
\label{lem:odd_alt}
For a positive integer $m$, let $n = 2m+1$.  Then, $\pi \in \AAA_n$ arises
by placing $``n"$ in either the first or the last position of elements 
of $\AAA_{2m}$, or placing $``n"$ in the even gaps of elements of 
$\AAA_{2m}$.  This way, after insertion, $``n"$ will appear in
an odd position in $\AAA_n$.  Or we place $``n"$ in the 
odd gaps of elements of $\SSS_{2m}-\AAA_{2m}$.  
This way, after insertion,
$``n"$ will appear in an even position in $\AAA_n$. 
\end{lemma}

\begin{lemma}
\label{lem:even_alt}
For a positive integer $m$, let $n = 2m$.  Then, $\pi \in \AAA_n$ arises
either by placing $``n"$ in the last position of elements of $\AAA_{n-1}$, or by 
placing $``n"$ in the first position of elements of 
$\SSS_{n-1} - \AAA_{n-1}$ or by placing $``n"$ in 
the odd gaps of elements of $\AAA_{n-1}$.  This way, after insertion,
$``n"$ will appear in an even position of $\AAA_n$.  Or, 
we could place $``n"$ in the even gaps of elements of $\SSS_{n-1}-\AAA_{n-1}$.
This way, after insertion, $``n"$ will appear in an odd position in $\AAA_n$. 
\end{lemma}

Our recurrence relation for the polynomials $A_n^+(s,t)$ 
and $A_n^-(s,t)$ depend on the parity of $n$ and so we bifurcate
the remaining part into two cases. 

\subsection{When $n = 2m+1$}

We begin with the following.

\begin{lemma}
\label{lem:different_sum}
For $0 \leq r \leq m-1$, let 
$S = [(\SSS_{2m}-\AAA_{2m}) \rightarrow \AAA_{2m+1}^{2r+1}]$ and 
$T = [\AAA_{2m} \rightarrow (\SSS_{2m+1}-\AAA_{2m+1})^{2m-2r-1}]$.  Then, the following is true.
\begin{equation}
\sum_{\pi \in S} t^{\des(\pi)}s^{\asc(\pi)} = \sum_{\pi \in T} t^{\des(\pi)}s^{\asc(\pi)}.
\end{equation}
\end{lemma}

\begin{proof} 
It is clear that if $\pi \in \AAA_{2m+1}^{2r+1}$ then $\pi' \in \SSS_{2m} - 
\AAA_{2m}$ and that if $\pi \in (\SSS_{2m+1}-\AAA_{2m+1})^{2r+1}$ then
$\pi' \in \AAA_{2m}$.
We give a bijection $f: S \mapsto T$  such that 
 $t^{\des(\pi)}s^{\asc(\pi)} = t^{\des(f(\pi))}s^{\asc(f(\pi))}$
for all $\pi \in S$.  Define
$$f(a_1,a_2,\ldots,a_{2r+1},2m+1,a_{2r+2}, \ldots,a_{2m})=a_{2r+2},\ldots,a_{2m},2m+1,a_1,a_2, \ldots,a_{2r+1}.$$

Let $\pi \in \SSS_{2m+1}$ and let $K, L$ be restrictions of $\pi$.
Define $\invpair_{\pi}(K,L) = |\{ (i,j): \pi_i \in K, \pi_j \in L, i < j, \pi_i > \pi_j\}|$.  
That is, $\invpair_{\pi}(K,L)$ is the number of inversions in $\pi$ between elements of 
$K$ and $L$. Clearly,
\begin{eqnarray}
\inv(\pi) & = & \inv(a_{2r+2}, \ldots ,a_{2m})+  2m-(2r+1)+\inv(a_1,\ldots,a_{2r+1})  \nonumber \\
	& & + \invpair_{\pi}([a_1,\ldots,a_{2r+1}],[a_{2r+2},\ldots,a_{2m}]) \\
\inv(f(\pi)) & = & \inv(a_{2r+2},\ldots,a_{2m}) + 2r + 1 + \inv(a_1,...,a_{2r+1}) \nonumber \\ 
	& & +  \invpair_{\pi}([a_{2r+2},\ldots,a_{2m}],[a_1,\ldots,a_{2r+1}] ) 
\end{eqnarray}
\begin{eqnarray}
\inv(\pi)-\inv(f(\pi)) & = & 
	\invpair_{\pi}([a_1,\ldots,a_{2r+1}], [a_{2r+2},\ldots,a_{2m}]) +
	2m \nonumber \\
	& & -2(2r+1) -\invpair_{\pi}([a_{2r+2},\ldots,a_{2m}], [a_1,\ldots,a_{2r+1}]) \\
	& \equiv & 1 \mbox{ (mod 2) }
\end{eqnarray}
The last line follows as 

$\invpair_{\pi}([a_{2r+2},\ldots,a_{2m}], [a_1,\ldots,a_{2r+1}]) -
\invpair_{\pi}([a_1,\ldots,a_{2r+1}], [a_{2r+2},\ldots,a_{2m}])$

$\equiv (2m-2r-1)*(2r+1)$ (mod 2) $\equiv$ 1 (mod 2).  
Thus, $\pi \in \AAA_{2m+1}$ iff $f(\pi) \in \SSS_{2m+1} - \AAA_{2m+1}$.  
Define $g:T \mapsto S$ by 
$$g(a_1,\ldots,a_{2m-2r-1},2m+1,a_{2m-2r},\ldots,a_{2m})
= a_{2m-2r},\ldots,a_{2m},2m+1,a_1,\ldots,a_{2m-2r-1}.$$  Clearly, $f$ is a
bijection with $f^{-1} = g$.  Further, for $\pi \in S$, it is 
clear that  $t^{\des(\pi)}s^{\asc(\pi)} = t^{\des(f(\pi))}s^{\asc(f(\pi))}$,
completing the proof.
\end{proof}

For $0 \leq r \leq m$, define 
\begin{eqnarray}
\label{eqn:P_defn}
P_r^{2m+1}(s,t) &= &\sum_{\pi \in S} t^{\des(\pi)}s^{\asc(\pi)} ,\\ 
\label{eqn:Q_defn}
Q_r^{2m+1}(s,t) &= & \sum_{\pi \in T} t^{\des(\pi)}s^{\asc(\pi)}.
\end{eqnarray}

Summing Lemma \ref{lem:different_sum} when $r$ runs from $0$ to $m-1$, 
we get the following.

\begin{corollary} 
\label{cor:sum}
For any odd natural number $2m+1$, 
the following equality holds.
$$\sum_{r=0}^{m-1} P_r^{2m+1}(s,t) = \sum_{r=0}^{m-1} Q_r^{2m+1}(s,t) $$ 
\end{corollary}  
 
With these preliminary results, we can now prove the following.

\begin{theorem}
\label{thm:main_odd}
Similar to the recurrence \eqref{eqn:rec_biv_eulerian}
for $A_n(s,t)$, we get the following two recurrences
for $A_{2m+1}^+(s,t)$ and $A_{2m+1}^-(s,t)$ respectively.
\begin{eqnarray} 
\label{eqn:n_odd_pi_even_recurrence}
A_{2m+1}^+(s,t) & = & (s+t)A_{2m}^+(s,t)+st D 
A_{2m}^+(s,t) \\
\label{eqn:n_odd_pi_odd_recurrence}
A_{2m+1}^-(s,t) & = & (s+t)A_{2m}^-(s,t)+st D
A_{2m}^-(s,t). 
\end{eqnarray} 
\end{theorem}
\begin{proof}
We prove \eqref{eqn:n_odd_pi_even_recurrence} first.  
Note from the proof of Theorem \ref{thm:foata_schut_Sn},
that the right hand side equals 
$\sum_{\pi \in \AAA_{2m+1}}t^{\des(\pi)}s^{\asc(\pi)}$  
where the summation is over $\pi$
with $2m+1$ coming in all the positions of $\AAA_{2m}$.

Recall 
$A_{2m+1}^+(s,t) = \sum_{\pi \in \AAA_{2m+1}}t^{\des(\pi)}s^{\asc(\pi)}$.  
We consider the contribution to the right hand side from $\pi \in \AAA_{2m+1}$ 
with the letter $2m+1$ occurring in position $\ell$ for all possible 
choices of $\ell$.  

Contribution from $\pi$ where $2m+1$ appears in the first or the last 
position  are accounted for by the terms
$sA_{2m}^+(s,t)$ and $tA_{2m}^+(s,t)$ respectively.  The remaining 
$\pi \in \AAA_{2m+1}$ arise when $2m+1$ appears in even positions of 
$\AAA_{2m}$ or when 
$\pi \in \AAA_{2m+1}$ arise when $2m+1$ appears in odd positions of 
$\SSS_{2m} - \AAA_{2m}$.
The contribution of these two possibilities are accounted for by the 
two terms $\sum_{r=1}^{m-1} \left( \sum_{\pi \in \AAA_{2m} \rightarrow 
\AAA_{2m+1}^{2r}} t^{\des(\pi)}s^{\asc(\pi)} \right)$ and 
$\sum_{r=1}^{m-1} \left( \sum_{\pi \in (\SSS_{2m} - \AAA_{2m}) \rightarrow 
\AAA_{2m+1}^{2r+1}} t^{\des(\pi)}s^{\asc(\pi)} \right)$ respectively.

By \eqref{eqn:P_defn}, the second term above 
$\sum_{r=1}^{m-1} \left( \sum_{\pi \in (\SSS_{2m} - \AAA_{2m}) \rightarrow 
\AAA_{2m+1}^{2r+1}} t^{\des(\pi)}s^{\asc(\pi)} \right)
= P_r^{2m+1}(s,t)$.  By Corollary \ref{cor:sum}, this equals
$Q_r^{2m+1}(s,t)$.  This equals the sum over all $\pi$
with $2m+1$ coming in all the positions of $\AAA_{2m}$,
completing the proof.  An identical proof shows
\eqref{eqn:n_odd_pi_odd_recurrence}.
\end{proof}

\subsection{When $n = 2m$}
The moves we make are similar to the case when $n=2m+1$, though there
are minor differences.  Hence, we omit some proofs.  We start with 
the following counterpart of Lemma 
\ref{lem:different_sum}.

\begin{lemma}
\label{lem:n_even_different_sum}
For $0 \leq r \leq m-1$, let 
$K = [\AAA_{2m-1} \rightarrow \AAA_{2m}^{2r+1}]$ and $L = 
[\AAA_{2m-1} \rightarrow (\SSS_{2m} - \AAA_{2m})^{2m-2r-2}]$. 
Let $U=[(\SSS_{2m-1}-\AAA_{2m-1}) \rightarrow 
(\SSS_{2m-1} - \AAA_{2m-1})^{2r+1}]$ and 
$V = [(\SSS_{2m-1}-\AAA_{2m-1}) \rightarrow \AAA_{2m-1}^{2m-2r-2}]$. 
Then, we assert the following.
\begin{eqnarray}
\label{eqn:n_even_ST}
\sum_{\pi \in K} t^{\des(\pi)}s^{\asc(\pi)} & = & \sum_{\pi \in L} t^{\des(\pi)}s^{\asc(\pi)}. \\
\label{eqn:n_even_PQ}
\sum_{\pi \in U} t^{\des(\pi)}s^{\asc(\pi)} & = & \sum_{\pi \in V} t^{\des(\pi)}s^{\asc(\pi)}.
\end{eqnarray}
\end{lemma}
\begin{proof}
A bijection similar to the one used to prove Lemma \ref{lem:different_sum} works.
We omit the details.
\end{proof}

For $n=2m$ and $0 \leq r \leq m-2$ define 

\begin{eqnarray}
\label{eqn:GH_defn}
G_r^{2m}(s,t) &= & \sum_{\pi \in K} t^{\des(\pi)}s^{\asc(\pi)},  \hspace{1 cm}
H_r^{2m}(s,t) =  \sum_{\pi \in L} t^{\des(\pi)}s^{\asc(\pi)}, \\
\label{eqn:IJ_defn}
I_r^{2m}(s,t) &= & \sum_{\pi \in U} t^{\des(\pi)}s^{\asc(\pi)}, \hspace{1 cm}
J_r^{2m}(s,t) =  \sum_{\pi \in V} t^{\des(\pi)}s^{\asc(\pi)}. 
\end{eqnarray}

Summing over Lemma \ref{lem:n_even_different_sum} when $0 \leq r \leq m-2$
gives the following corollary.

\begin{corollary} 
\label{cor:even_sum}
For any even natural number $2m$, 
the following two equalities hold.
$$\sum_{r=0}^{m-1} G_r^{2m}(s,t) = \sum_{r=0}^{m-1} H_r^{2m}(s,t) \hspace{5 mm} \mbox{  and  } \hspace{5 mm}
\sum_{r=0}^{m-1} I_r^{2m}(s,t) = \sum_{r=0}^{m-1} J_r^{2m}(s,t).$$
\end{corollary}

\begin{theorem}
\label{thm:main_even}
Similar to the recurrence \eqref{eqn:rec_biv_eulerian}
for $A_n(s,t)$, we get the following recurrences 
for $A_{2m}^+(s,t)$ and $A_{2m}^-(s,t)$.
\begin{eqnarray} 
\label{eqn:n_even_pi_even_recurrence}
  A_{2m}^+(s,t) & = & s A_{2m-1}^+(s,t) + t A_{2m-1}^-(s,t) + \frac{1}{2}st D
	A_{2m-1}(s,t) \\
\label{eqn:n_even_pi_odd_recurrence}
  A_{2m}^-(s,t) & = & s A_{2m-1}^-(s,t) + t A_{2m-1}^+(s,t) + \frac{1}{2}stD 
A_{2m-1}(s,t) 
\end{eqnarray} 
\end{theorem}
\begin{proof}
We prove \eqref{eqn:n_even_pi_even_recurrence} first.
Recall $A_{2m}^+(s,t) = \sum_{\pi \in \AAA_{2m}}t^{\des(\pi)}s^{\asc(\pi)}$.  
Consider the contribution to the right hand side from $\pi \in \AAA_{2m}$ 
with the letter $2m$ occurring in position $\ell$ for all possible 
choices of $\ell$.  

Contribution from $\pi$ where $2m$ appears in the first or the last 
position  are accounted for by the terms
$tA_{2m-1}^-(s,t)$ and $sA_{2m-1}^+(s,t)$ respectively.  The remaining 
$\pi \in \AAA_{2m}$ arise when $2m$ appears in even positions of 
$\AAA_{2m-1}$ or when $2m$ appears in odd positions of 
$\SSS_{2m-1} - \AAA_{2m-1}$.
By Corollary \ref{cor:even_sum}, such permutations contribute 
$\frac{1}{2} 
\sum_{r=0}^{m-2} 
\left( 
G_r^{2m}(s,t) + H_r^{2m}(s,t) + 
I_r^{2m}(s,t) + J_r^{2m}(s,t) \right)$, completing the proof.
\end{proof}

Tanimoto gave the following recurrence for the numbers $a_{n,k}^+$ and 
$a_{n,k}^-$.  These easily follow from Theorems 
\ref{thm:main_odd} and \ref{thm:main_even}.

\begin{corollary}[Tanimoto]
\label{cor:tanimoto_rec_2m}
Let $n=2m+1$ be an odd positive integer.  Then,
$$
a_{n,i}^+ =  (n-k)a_{n-1,k-1}^+  + (k+1)a_{n-1,k}^+ \hspace{3 mm}
\mbox{ and } \hspace{3 mm}
a_{n,i}^- =  (n-k)a_{n-1,k-1}^-  + (k+1)a_{n-1,k}^-.$$

Let $n=2m$ be an even positive integer.  Then,

\vspace{-5 mm}

\begin{eqnarray*}
2a_{n,i}^+ &=& (n-k+1)a_{n-1,k-1}^+  + k a_{n-1,k}^+ + (n-k-1)a_{n-1,k-1}^- + (k+2)a_{n-1,k}^- \\
2a_{n,i}^- &=& (n-k+1)a_{n-1,k-1}^-  + k a_{n-1,k}^- + (n-k-1)a_{n-1,k-1}^+ + (k+2)a_{n-1,k}^+
\end{eqnarray*}
\end{corollary}

\section{Gamma positivity of $A_n^+(s,t)$ when $n \equiv 0,1$ (mod 4)}
\label{sec:main_thm}

We are now ready to prove one of our main results.

\begin{proof} (Of Theorem \ref{thm:An_01_mod4})
By Lemma \ref{lem:01_mod_4}, the polynomials $A_n^+(s,t)$ and $A_n^-(s,t)$ are
palindromic iff $n \equiv 0,1$ (mod 4).  Our proof is by induction on 
$n$ with the base case being $n=4$.  In this case, it is easy to see
that $A_4^+(s,t) = s^3+5s^2t+5st^2+t^3$ and that $A_4^-(s,t) = 6s^2t+6st^2$.  Both
these are clearly gamma positive. 

When $n \equiv 1$ (mod 4), the recurrence in Theorem \ref{thm:main_odd} is identical
to Foata and Sch{\"u}tzenberger's recurrence in Theorem  \ref{thm:foata_schut_Sn}. Thus, the
same proof of Theorem \ref{thm:foata_schut_Sn} shows that both $A_n^+(s,t)$ and
$A_n^-(s,t)$ are  gamma positive with the same center of symmetry.  
To go from $n= 4k+1$ to $n=4k+4$, we use both Theorem 
\ref{thm:main_even} and Theorem \ref{thm:main_odd}.  
Recall $D$ is the operator $\displaystyle \left( \frac{d}{ds} + \frac{d}{dt}
\right)$.  The following is easy to see.  

\begin{eqnarray}
A_{4k+4}^+(s,t) & = & L_1(s,t) A_{4k+1}^+(s,t) + L_2(s,t) A_{4k+1}^-(s,t)
+ L_3(s,t) D A_{4k+1}^+(s,t) \label{eqn:4m+4_jump} \\ 
& & + L_4(s,t)  DA_{4k+1}^-(s,t) + L_5(s,t) D^2 A_{4k+1}(s,t) + L_6(s,t) 
D^3 A_{4m+1}(s,t)  \nonumber 
\end{eqnarray}

where the following table lists $L_i(s,t)$ and its center of 
symmetry.

$
\begin{array}{l|r}
f(s,t) & \cnos(f(s,t)) \\ \hline
L_1(s,t) = (s+t)^3 + 2ts(s+t) &  3/2 \\ \hline
L_2(s,t) = 6ts(s+t)  &  3/2 \\ \hline 
L_3(s,t) = 2(st)^2 + 4st(s+t)^2 &   2 \\ \hline
L_4(s,t) = 3st(s+t)^2 + 6(st)^2&  2\\ \hline
L_5(s,t) = 3(st)^2(s+t) & 5/2 \\ \hline
L_6(s,t) = 1/2(st)^3 & 3 
\end{array}
$

\vspace{2 mm}

Using Corollary \ref{cor:gamma_nonneg}, we get that each of the 
six terms in \eqref{eqn:4m+4_jump} has center of symmetry $2m+3/2$.  Since 
the terms all have the same center of symmetry, the polynomial
$A_{4k+4}^+(s,t)$ is gamma positive.  An indentical proof works for 
$A_{4k+4}^-(s,t)$ and one can check that both polynomials have the same
center of symmetry.  The proof is complete.
\end{proof}

Several interpretations for the numbers $\gamma_{n,i}$ as cardinalities
of appropriate sets are known (see 
Athanasiadis \cite[Theorem 2.1]{athanasiadis-survey-gamma-positivity}).  
We end this section with the following.

\begin{question}
\label{qn:gamma_interpret_01_mod4}
Let $n \equiv 0,1$ (mod 4).  By Theorem \ref{thm:An_01_mod4}, both 
$A_n^+(s,t) = \sum_{i=0}^{\nmhalf} \gamma_{n,i}^+(st)^i(s+t)^{n-1-2i}$
and 
$A_n^-(s,t) = \sum_{i=0}^{\nmhalf} \gamma_{n,i}^-(st)^i(s+t)^{n-1-2i}$.
Can we get an interpretation for the gamma coeffients $\gamma_{n,i}^+$ or
$\gamma_{n,i}^-$?
\end{question}

\section{When $n \equiv 2$ (mod 4)}
When $n \equiv 2$ (mod 4), the polynomials $A_n^+(t)$ and 
$A_n^-(t)$ are not palindromic and so cannot be expressed in 
the gamma basis.  Nonetheless, we show that both the above polynomials
can be written as a sum of two gamma positive  polynomials 
with different centers of symmetry.  We note that this result only 
holds for the univariate polynomials $A_n^+(t)$ and $A_n^-(t)$ and 
not for the bivariate counterparts $A_n^+(s,t)$ and $A_n^-(s,t)$.

\subsection{Sum of gamma positive polynomials}
\begin{proof} (Of Theorem \ref{thm:An_more_gamma_positive})
We prove the result for $A_{4m+2}^+(t)$.  An identical proof 
works for $A_{4m+2}^-(t)$.  As $n=4m+2$, by Theorem \ref{thm:main_even} 
we have

\begin{eqnarray}
  A_{4m+2}^+(s,t)  & =  & s A_{4m+1}^+(s,t) + t A_{4m+1}^-(s,t) 
	+ \frac{1}{2}st D A_{4m+1}(s,t) \nonumber \\
\label{eqn:n_even_pi_even_recurrence_univariate}
  A_{4m+2}^+(t)  & =  &  A_{4m+1}^+(t) + t A_{4m+1}^-(t) 
	+ \left( \frac{1}{2}st D A_{4m+1}(s,t) \right)|_{s=1} 
\end{eqnarray}

Let $p(s,t) = \frac{1}{2}st D A_{4m+1}(s,t)$.
By Corollary \ref{cor:gamma_nonneg}, $p(s,t)$ is gamma positive 
with $\cnos(p(s,t)) = 2m-1/2+1$.  Further, 
by Corollary \ref{cor:gamma_coeff_oddness}, all coefficients of 
$A_{4m+1}(s,t)$ except the coefficient of $(st)^0(s+t)^{4m}$
are even.  
As the exponent $4m$ will appear
in $D(s+t)^{4m}$ that coefficient will also be even.  Hence
$p(s,t)$ and 
$p(t) = p(s,t)|_{s=1}$ are  gamma positive
with integral coefficients in the gamma basis and we have 
$\cnos(p(t)) = 2m+1/2$.

Further, it is simple to see that $p(t)$ has odd length and thus by 
Lemma \ref{lem:more_gamma_positive},  it can be written as 
$p(t) = p_1(t) + p_2(t)$ with $\cnos(p_1(t)) = 2m$ and 
$\cnos(p_2(t)) = 2m+1$.  Thus,
\eqref{eqn:n_even_pi_even_recurrence_univariate} becomes
\begin{eqnarray}
\label{eqn:sum_4m+2_one}
  A_{4m+2}^+(t) &=&  A_{4m+1}^+(t) + t A_{4m+1}^-(t) + p_1(t) + p_2(t) \\
\label{eqn:sum_4m+2_two}
&=  & w_1(t) + w_2(t) 
\end{eqnarray}
where $w_1(t) = A_{4m+1}^+(t) + p_1(t)$ has $\cnos(w_1(t)) = 2m$ 
and $w_2(t) = tA_{4m+1}^-(t) + p_2(t)$ has $\cnos(w_2(t)) = 2m+1$.
The proof is complete.
\end{proof}

The bivariate Eulerian polynomials restricted to the alternating group are not gamma 
positive in the stronger sense.  We postpone the example to after the next subsection
to avoid repeating the data (see Remark \ref{rem:bivariate_false}).

\subsection{A curious identity} 
When $n \equiv 2$ (mod 4), Theorem \ref{thm:An_more_gamma_positive} states
that both the polynomials $A_n^+(t)$ and $A_n^-(t)$ 
can be written as a sum of two gamma positive polynomials while Theorem
\ref{thm:foata_schut_Sn} states that $A_n(t)$ is gamma positive.
Below, we see how these four gamma-positive polynomials arising from
$A_n^+(t)$ and $A_n^-(t)$ add up to give a single gamma positive polynomial.
We illustrate this with an example when $n=6$.  It is easy
to see that

\vspace{-5 mm}
$$A_6(t)  =  1+57t+302t^2+302t^3+57t^4+t^5 \hspace{1 cm} \mbox{ and that }$$

\vspace{2 mm}

\begin{tabular}{rl|rl}
$A_6^+(t)$ & $= 1+29t + 147t^2 + 155t^3 + 28t^4$  & $A_6^-(t)$ 	
& $=28t+155t^2+147t^3+29t^4 + t^5$ \\
$f_1(t)$ & $= 1+29t+120t^2+29t^3+t^4$ & $g_1(t)$ & $=  t+29t^2+120t^3+29t^4+t^5$ \\
$f_2(t)$ & $= 27t^2+126t^3 + 27t^4$   & $g_2(t)$ & $= 27t+126t^2 + 27t^3$ 
\end{tabular}

\vspace{2 mm}

In the gamma basis, we have

\vspace{2 mm}

$\begin{array}{l|l|l|l}
  A_6(t) & (1+t)^5 & \red{52}t(1+t)^3 & \blue{136}t^2(1+t) \\  \hline 
  g_1(t) & t(1+t)^4 & \red{25}t^2(1+t)^2 & \blue{64}t^3 \\
  g_2(t) &  & \red{27}t(1+t)^2 & \blue{72}t^2 \\ \hline
f_1(t) & (1+t)^4 & 25t(1+t)^2 & 64t^2 \\
f_2(t) &  & 27t^2(1+t)^2 & 72t^3
\end{array}
$

\vspace{4 mm}

If $A_{4m+2}^+(t) = f_1(t) + f_2(t)$ then, each gamma coefficient of $A_{4m+2}(t)$ 
is the sum of a gamma 
coefficient of $f_1(t)$ and a gamma coefficient of $f_2(t)$ (and 
independently a sum of two gamma coefficients, one each of $g_1(t)$ and
$g_2(t)$).  We prove this below.  

Let $\gamma_{n,i}^{\SSS}$ denote the $i$-th gamma coefficient when
$A_n(t)$ is expressed in the $\Gamma$ basis.  Similarly, let 
$A_{4m+2}^+(t) = f_1(t) + f_2(t)$ with 
$f_1(t) = \sum_{i=0}^{2m} \gamma_{4m+2,i}^{\AAA,1} t^i (1+t)^{4m -2i}$ and
with 
$f_2(t) = \sum_{i=0}^{2m+1} \gamma_{4m+2,i}^{\AAA,2} t^i (1+t)^{4m -2i +2}$.
When $n \equiv 0,1$ (mod 4), let 
$A_n^+(t) = \sum_{i=0}^{\nmhalf} \gamma_{n,i}^{\AAA} t^i (1+t)^{n-1-2i}$
and let 
$A_n^-(t) = \sum_{i=0}^{\nmhalf} \gamma_{n,i}^{\SSS - \AAA} t^i (1+t)^{n-1-2i}$.

\begin{lemma}
\label{lem:gamma_coeffs_recursive}
Let $n=4m+2$ where $m$ is a positive integer.  With the above notation,
\begin{equation}
\label{eqn:gamma_add_up}
\gamma_{4m+2,i-1}^{\SSS} = \gamma_{4m+2,i-1}^{\AAA,1} + \gamma_{4m+2,i}^{\AAA,2}.
\end{equation}
\end{lemma}
\begin{proof}
Equation \eqref{eqn:sum_4m+2_two} in the proof of Theorem 
\ref{thm:An_more_gamma_positive} gives $A_{4m+2}^+(t) = f_1(t) + f_2(t)$.  
The proof of Theorem \ref{thm:An_more_gamma_positive} also shows that

\begin{eqnarray}
f_1(t) & = & 
\sum_{i} \frac{1}{2} \gamma_{4m+1,i}^{\SSS} it^i(1+t)^{4m-2i} + 
\sum_{i} \gamma_{4m+1,i}^{\SSS} (4m-2i)t^{i+1}(1+t)^{4m-2i-2} \nonumber \\
\label{eqn:gamma1_2mod4}
 & & + A_{4m+1}^+(t)  \\
f_2(t) & = & 
\sum_{i} \frac{1}{2} \gamma_{4m+1,i}^{\SSS} it^{i+1}(1+t)^{4m-2i-2} + 
\sum_{i} \gamma_{4m+1,i}^{\SSS} (4m-2i)t^{i+2}(1+t)^{4m-2i-4} \nonumber \\
\label{eqn:gamma2_2mod4}
 & & + tA_{4m+1}^-(t)  
\end{eqnarray}

Thus, \begin{eqnarray*}
\gamma_{4m+2,i}^{\AAA,1}& =& \frac{i}{2}\gamma_{4m+1,i}^{\SSS} + (4m-2(i-1))
        \gamma_{4m+1,i-1}^{\SSS}+ \gamma_{4m+1,i}^{\AAA} ,\\
        \gamma_{4m+2,i}^{\AAA,2} & = &
        \frac{i-1}{2}\gamma_{4m+1,i-1}^{\SSS} + (4m-2(i-2))
        \gamma_{4m+1,i-2}^{\SSS}+ \gamma_{4m+1,i-1}^{\SSS-\AAA}
        \end{eqnarray*}
        
Hence, 
$\gamma_{4m+2,i-1}^{\AAA,1} + \gamma_{4m+2,i}^{\AAA,2}= [4m-2(i-2)]\gamma_{4m+1,i-1}^{\SSS}+ i\gamma_{4m+1,i}^{\SSS}= \gamma_{4m+2,i-1}^{\SSS}
$, completing the proof. 
\end{proof}

\begin{remark}
\label{rem:bivariate_false}
When summed over $\AAA_n$, the bivariate Eulerian polynomials are not gamma positive in
the stronger sense.  When $n=6$, our proof gives
\begin{eqnarray*}
A_6^+(s,t) & = & s^5+29s^4t + 147s^3t^2 + 155s^2t^3 + 28st^4 \\
f_1(s,t) & = & s(s+t)^4 + 25s^2t(s+t)^2 + 64s^3t^2\\
f_2(s,t) & = & 27st^2(s+t)^2 + 72s^2t^3.
\end{eqnarray*}

Thus, though $A_6^+(t)$ is a sum
of two gamma positive polynomials, $A_6^+(s,t)$ is not.
\end{remark}

\section{When $n\equiv 3$ (mod 4)}
Our main result of this section is the following counterpart
of Theorem  \ref{thm:An_more_gamma_positive}.

\begin{proof} (Of Theorem \ref{thm:An_still_more_gamma_positive})
Recall from \eqref{thm:main_odd}  that
$A_{4m+3}^+(s,t)  =  (s+t)A_{4m+2}^+(s,t) + st DA_{4m+2}^+(s,t)$.
By \eqref{thm:main_even}, 

\begin{eqnarray*}
(s+t)A_{4m+2}^+(s,t)|_{s=1} & = & \left \{ (s+t) \left(
sA_{4m+1}^+(s,t) + tA_{4m+1}^-(s,t) + \frac{st}{2} D A_{4m+1}(s,t)  
\right) \right \}|_{s=1} \\
& = & (1+t)\left[ A_{4m+1}^+(t) + tA_{4m+1}^-(t) \right] + 
\left \{ \frac{st}{2}(s+t)DA_{4m+1}(s,t) \right \}|_{s=1} 
\end{eqnarray*}

Let 
\begin{equation}
\label{eqn:first_3_mod_4}
p_1(t) = (1+t)A_{4m+1}^+(t),  p_2(t) = t(1+t)A_{4m+1}^-(t) \mbox{ and }
p_3(t) = \left\{ \frac{st}{2} (s+t)D A_{4m+1}(s,t)\right\}|_{s=1}
\end{equation}

By 
Corollary \ref{cor:gamma_nonneg}, 
the three polynomials $p_1(t)$, $p_2(t)$ and 
$p_3(t)$ have respective centers of symmetry $2m+\frac{1}{2}$, 
$2m + \frac{3}{2}$ and 
$2m+1$.

\begin{eqnarray}
\nonumber
stDA_{4m+2}^+(s,t)|_{s=1} & = & stD \left( 
sA_{4m+1}^+(s,t) + tA_{4m+1}^-(s,t) + \frac{st}{2} D A_{4m+1}(s,t)  
 \right)|_{s=1} \\
\nonumber
& = &  \left( s^2tD A_{4m+1}^+(s,t) + stA_{4m+1}^+(s,t) + st^2D A_{4m+1}^-(s,t) 
+ stA_{4m+1}^-(s,t) \right. \\
\nonumber
& & \left. + stD (\frac{1}{2} stD A_{4m+1}(s,t) \right)|_{s=1} \\
\nonumber
& = & s^2tD A_{4m+1}^+(s,t)|_{s=1} + t A_{4m+1}^+(t) + st^2DA_{4m+1}^-(s,t)|_{s=1} \\
\label{eqn:second_3_mod_4}
 & & + tA_{4m+1}^-(s,t) + st D (\frac{1}{2} st D A_{4m+1}(s,t))|_{s=1}
\end{eqnarray}

Define 
\begin{eqnarray}
\label{eqn:next1_3_mod_4}
p_{4m+3}^1(t) & = & (1+t)A_{4m+1}^+(t) + \left( s^2tD A_{4m+1}^+(s,t)\right)|_{s=1} \\
\nonumber
p_{4m+3}^2(t) & = & tA_{4m+1}^+(t) + tA_{4m+1}^-(t) + 
	\left( (s+t)\frac{st}{2} D A_{4m+1}(s,t)\right)|_{s=1} \\
\label{eqn:next2_3_mod_4}
	&  & + \left( stD \frac{st}{2} D A_{4m+1}(s,t)  \right)|_{s=1} \\
\label{eqn:next3_3_mod_4}
p_{4m+3}^3(t) & = & (1+t)tA_{4m+1}^-(t) + \left( st^2 D A_{4m+1}^-(s,t)\right)|_{s=1} 
\end{eqnarray}

By 
Corollary \ref{cor:gamma_nonneg}, 
$\cnos(p_{4m+3}^1(t)) = 2m+ \frac{1}{2}$, 
$\cnos(p_{4m+3}^2(t)) = 2m+ 1$ and 
$\cnos(p_{4m+3}^3(t)) = 2m+ \frac{3}{2}$.  Clearly 
adding \eqref{eqn:first_3_mod_4} and \eqref{eqn:second_3_mod_4}
is equivalent to adding 
\eqref{eqn:next1_3_mod_4}, \eqref{eqn:next2_3_mod_4} and 
\eqref{eqn:next3_3_mod_4}.  Thus, $A_{4m+3}^+(t)$ can be written 
as a sum of three gamma positive polynomials.
A very similar proof can be given to show that $A_{4m+3}^-(t)$
can be written as a sum of three gamma positive polynomials.
The proof is complete. 
\end{proof}

\section{Type-B Coxeter Groups}

Recall that $\BB_n$ is the set of permutations $\pi$ of $\{-n, -(n-1),
\ldots, -1, 1, 2, \ldots n\}$ that satisfy $\pi(-i) = -\pi(i)$. 
$\BB_n$ is referred to  as the hyperoctahedral group or the group of 
signed permutations on $[n]$.   
For $\pi \in \BB_n$, for $1 \leq i \leq n$, we alternatively denote $\pi(i)$ as $\pi_i$.  
For $\pi \in \BB_n$, define $\Negs(\pi) = \{\pi_i : i > 0, \pi_i < 0 \}$ be 
the set of elements which occur with a negative sign.  
Define $\inv_B(\pi)=  | \{ 1 \leq i < j \leq n : \pi_i > \pi_j \} |+  
| \{ 1 \leq i < j \leq n : -\pi_i > \pi_j \} +|\Negs(\pi)| $.  We refer
to $\inv_B(\pi)$ alternatively as the {\it length of $\pi \in \BB_n$}.
Let $\BB^+_n \subseteq \BB_n$ 
denote the subset of even length elements of $\BB_n$.  
Let $\BB_n^- = \BB_n - \BB_n^+$. For all $\pi \in \BB_n $, 
let $ \pi_0=0$ and 
let $\des_B(\pi)= |\{ i \in \{0,1,2\ldots ,n-1\}: \pi_i > \pi_{i+1}\}| $ and 
$\asc_B(\pi)= |\{ i \in \{0,1,2\ldots ,n-1\}: \pi_i <  \pi_{i+1}\}|$. 
 Define

\vspace{-5 mm}

\begin{eqnarray}
\label{eqn:Bn} 
B_n(t) & = & \sum_{\pi \in \SSS_n} t^{\des_B(\pi)} 
\mbox{ 
\hspace{ 3 mm}
and 
\hspace{ 3 mm}
}
B_n(s,t)  =  \sum_{\pi \in \SSS_n} t^{\des_B(\pi)} s^{\asc_B(\pi)} =\sum_{k=0}^{n}b_{n,k}t^ks^{n-k} \\  
\label{eqn:BAn} 
B_n^+(t) & = & \sum_{\pi \in \BB_n^+} t^{\des_B(\pi)} 
\mbox{ 
\hspace{ 3 mm}
and 
\hspace{ 3 mm}
}
B_n^+(s,t)  =  \sum_{\pi \in \BB_n^+} t^{\des_B(\pi)} s^{\asc_B(\pi)} =\sum_{k=0}^{n}b_{n,k}^+t^ks^{n-k} \\
\label{eqn:BnAn} 
B_n^-(t) & = & \sum_{\pi \in \BB_n^-} t^{\des_B(\pi)} 
\mbox{ 
\hspace{ 3 mm}
and 
\hspace{ 3 mm}
}
B_n^-(s,t)  =  \sum_{\pi \in \BB_n^-} t^{\des_B(\pi)} s^{\asc_B(\pi)}=
\sum_{k=0}^{n}b_{n,k}^-t^ks^{n-k} 
\end{eqnarray}

The polynomial $B_n(t)$ is called the type-B Eulerian polynomial.
Hyatt in \cite{hyatt-recurrences_eulerian_typeBD} has given 
recurrences for the type-B and type-D Eulerian polynomials.  
We could not find a  type-B counterpart of Equation \eqref{eqn:rec_biv_eulerian}
to the best of our knowledge.  We start with the following

\begin{theorem}
	\label{thm:foata_schutzenberger_typeB}
	With $B_n(s,t)$ as defined in \eqref{eqn:Bn}, we have 
	\vspace{-3 mm}
	\begin{equation}
	B_n(s,t) = \sum_{i=0}^{\nhalf} \gamma_{n,i}^B (st)^i (s+t)^{n-2i}
	\end{equation}
	\vspace{-3 mm}
	where the $\gamma_{n,i}^B$ are non-negative integers for all 
	$n,i$.
\end{theorem}

\begin{proof}
Recall $\displaystyle D = \left( \frac{d}{ds} + \frac{d}{dt} \right)$.
We have the following version of \eqref{eqn:rec_biv_eulerian}.

\begin{equation}
\label{eqn:rec_biv_eulerian_B}
B_{n+1}(s,t) = (s+t)B_n(s,t) + 2st D B_n(s,t).
\end{equation}
	
We observe  that the term $(s+t)B_n(s,t)$ is the contribution of all 
those permutations $\pi \in \BB_n$ where the 
letter `$n+1$' or $\ob{n+1}$, is at the last position.  It is 
easy to see that 
$\displaystyle 2st
DB_n(s,t)$ 
is the contribution of all $\pi \in \BB_n$ in which 
the letter `$n+1$' or $\ob{n+1}$ appears in position $r$
for $1 \leq r \leq n$.

It is straightforward to check that $B_1(s,t) = s+t$.  Hence, by induction and Corollary 
\ref{cor:gamma_nonneg}, we get that $B_{n+1}(s,t)$ is gamma-positive.  
Further, it is simple to see that the following recurrence is satisfied by 
the coefficients $\gamma_{n+1,i}^B$.

\vspace{-3 mm}

\begin{equation}
\label{eqn:gamma_Bn}
\gamma^B_{n+1,i} = (2i+1)\gamma^B_{n,i} + 4[n-2(i-1)]\gamma^B_{n,i-1}.
\end{equation}
	
Alternatively, the above recurrence together with the initial conditions 
$\gamma^B_{1,0} = 1$ or $\gamma^B_{2,0} = 1, \gamma^B_{2,1} = 4$ 
settles non-negativity of $\gamma^B_{n,i}$ for all $n,i$.
\end{proof}

Similar to Lemma \ref{lem:01_mod_4}, the following result shows 
palindromicity of the polynomials $B_n(s,t)$.

\begin{lemma}
\label{lem:0_mod_2 for B_n^+}
For $n\geq 2$, the polynomials
$B_n^+(s,t)$ and $B_n^-(s,t)$ are palindromic iff 
$n \equiv 0$ mod 2.
\end{lemma}
\begin{proof}
Let $\pi = x_1,x_2,\ldots,x_n \in \BB_n$.  Define $f: \BB_n \mapsto \BB_n$ 
by $f(\pi) = \ob{x_1},\ob{x_2}, \ldots, \ob{x_n}$.  
Clearly, $f$ is a bijection.  It is easy to see that 
(or see \cite[Lemma 3]{siva-sgn_exc_hyp})
that flipping the sign of a single $x_i$ changes the parity of the number 
of type-B inversions.
Hence, if $n \equiv 0$ mod 2, then 
$\pi \in \BB_n^+$ iff $f(\pi) \in \BB_n^+$.  

If $n \equiv 1$ mod 2, then the map $f$ flips the parity of $\inv_B(\pi)$ and 
hence $b_{n,k}^+ = b_{n,n-k}^-$.  Further,
$b_{n,0}^+ = 1,b_{n,n}^+ = 0$ and 
$b_{n,n}^- =1, b_{n,0}^- = 0$ and hence $B_n^+(s,t)$ and 
$B_n^-(s,t)$ are 
not palindromic.
\end{proof}

Define
\begin{equation}
\label{eqn:sgn_des_type_B}
\SB_n(s,t)  =  \sum_{\pi \in \BB_n }(-1)^{\inv_B(\pi)}t^{\des_B(\pi)}s^{\asc_B(\pi)} 
\end{equation}
Clearly, $B_n^+(s,t) = \frac{1}{2} \left( B_n(s,t) + \SB_n(s,t) \right)$.  Motivated by
this, we determine $\SB_n(s,t)$ in the next Subsection.

\subsection{Enumeration of descent with signs in $\BB_n$}

Let $a_1,a_2,\ldots,a_n$ be $n$ distinct positive integers with 
$a_1 < a_2 < \cdots < a_n$. Let $\BB_{\{a_1,a_2,\ldots,a_n \}}$ 
be the Type-B group of $2^nn!$ permutations on the letters 
$a_1,a_2,\ldots,a_n$.  Clearly, when 
$a_i = i$ for all $i$, we get $\BB_n = \BB_{ \{1,2,\ldots,n \}}$. 
We write $\pi \in \BB_{\{a_1,a_2, \ldots, a_n\} }$ in two line
notation with $a_1, a_2, \ldots, a_n$ above and $\pi_{a_i}$ below
$a_i$.  Thus, $\pi_{a_i}$ is defined for all $i$.
 For any permutation $\pi \in  \BB_{\{a_1,a_2,\ldots,a_n \}}$
define $\inv_B(\pi)= | \{ 1 \leq i < j \leq n : \pi_{a_i} > \pi_{a_j} \} |+  
| \{ 1 \leq i < j \leq n : - \pi_{a_i} > \pi_{a_j} \} |   \} +|\Negs(\pi)|$.
Let $\BB_{ \{a_1,a_2,\ldots,a_n \} }^+$ 
denote the subset of even length elements of $\BB_{\{a_1,a_2,\ldots,a_n \}}$ and 
$\BB_{ \{a_1,a_2,\ldots,a_n \} }^- = \BB_{ \{a_1,a_2,\ldots,a_n \}}- \BB_{ \{a_1,a_2,\ldots,a_n \}}^+$. 
Let $a_0 = 0$ and $\pi(0)=0$ for all $\pi$.  Define  
$\des_B(\pi)= |\{ i \in \{0,1,2\ldots ,n-1\}: \pi_{a_i} > \pi_{a_{i+1}}\}| $ and 
$\asc_B(\pi)= |\{ i \in \{0,1,2\ldots ,n-1\}: \pi_{a_i} <  \pi_{a_{i+1}}\}|$. 
Also define $\pos_{a_i}(\pi)=k$ if $|\pi_{a_k}| =a_i$.  
Define 
 \begin{eqnarray}
\SB_{\{a_1,a_2,\ldots,a_n\}}(s,t,u)=\sum_{\pi \in \BB_{\{a_1,a_2,\ldots,a_n \}} }
(-1)^{\inv_B(\pi)}t^{\des_B(\pi)}s^{\asc_B(\pi)}u^{\pos_{a_n}(\pi)}
\end{eqnarray}
We start with the following.

\begin{lemma} 
\label{lem:half_pos_neg_sum}
For positive integers $a_1<a_2<\ldots<a_{n-1}<a_n$, let 

\vspace{2 mm}

$S = \BB^+_{{{\{a_1,a_2,\ldots,a_{n}\} }},\pos_{a_{n-1}}(\pi) = n, \pos_{a_n}(\pi)=n-1}$ and 
$T = \BB^-_{{{\{a_1,a_2,\ldots,a_{n}\} }},\pos_{a_{n-1}}(\pi) = n, \pos_{a_n}(\pi)=n-1}$.

\vspace{2 mm}

Then, the following holds:

\begin{equation}
\label{eqn:an_at_penultimate, an-1_at_last}
\sum_{\pi \in S} (-1)^{\inv_B(\pi)}t^{\des_B(\pi)}s^{\asc_B(\pi)} 
u^{\pos_{a_n}(\pi)} =
\sum_{\pi \in T}
(-1)^{\inv_B(\pi)}t^{\des_B(\pi)}s^{\asc_B(\pi)}
u^{\pos_{a_n}(\pi)}
\end{equation}
\end{lemma}
\begin{proof}
Define $h:S \mapsto T$ by $h(\pi_{a_1},\pi_{a_2},\ldots,\pi_{a_{n-1}},\pi_{a_n})=\pi_{a_1},\pi_{a_2},\ldots,
\pi_{a_{n-1}},\ob{\pi_{a_n}}$. It is clear that $h$ is a bijection that 
preserves descents, $\pos_{a_n}(\pi)$ and flips the inversion number, 
completing the proof.
\end{proof}

\begin{lemma}
\label{lem:signed_sum_zero}
For all positive integers $a_1<a_2<\ldots<a_n$, 
\begin{eqnarray}
\label{eqn:an_not_at_last_position}
\sum_{\pi \in \BB_{{\{a_1,a_2,\ldots,a_{n} \}}},
\pos(a_{n})(\pi) \neq n  }(-1)^{\inv_B(\pi)}t^{\des_B(\pi)}
s^{\asc_B(\pi)} u^{\pos_{a_n}(\pi)}= 0
\end{eqnarray}
\end{lemma}
\begin{proof}
We induct on $n$. When $n=1$, the sum is clearly $0$.  For any set of  
$n-1$ distinct integers, let the sum be $0$.  Note that the set 
$\{\pi \in \BB_{{\{a_1,a_2,\ldots,a_{n} \}}},
\pos_{a_n}(\pi) \neq n \}$ is the disjoint union of the following three sets 
$S^1 = \{\pi \in \BB_{{\{a_1,a_2,\ldots,a_{n}\} }}, \pos_{a_{n-1}}(\pi) = n, \pos_{a_n }(\pi) = n-1  \}$ , 
$S^2 = \{\pi \in \BB_{{\{a_1,a_2,\ldots,a_{n} \}}}, \pos_{a_{n-1}}(\pi) = n, \pos_{a_n}(\pi) ) \neq n-1 \}$ and 
$S^3 = \{\pi \in \BB_{{\{a_1,a_2,\ldots,a_{n-2},a_{n-1} \}}},
\pos_{a_{n-1}}(\pi)  \neq n, 	\pos_{a_{n}}(\pi)  \neq n \}$.  
By Lemma \ref{lem:half_pos_neg_sum}, enumeration over $S^1$ gives the required sum to be zero.
We need the following two claims when enumeration is over $S^2$ and $S^3$ respectively.

{\bf Claim 1:} For all $n$  positive integers $a_1<a_2<\ldots<a_n$, the following holds:  
\begin{eqnarray}
\label{eqn:an-1_at_last_an_not_at_penultimate}
\sum_{\pi \in S^2} (-1)^{\inv_B(\pi)} t^{\des_B(\pi)}s^{\asc_B(\pi)} u^{\pos_{a_n}(\pi)}= 0
\end{eqnarray}
{\bf Proof of Claim 1:}
By induction on the integers $a_1<a_2<\ldots<a_{n-2}<a_n$ we get 
\begin{eqnarray*}\label{eqn:an not at last}
\sum_{\pi \in 
\BB^+_{{\{a_1,a_2,\ldots,a_{n-2},a_{n} \}}},\pos_{a_n}(\pi) \not= n-1}
 (-1)^{\inv_B(\pi)}t^{\des_B(\pi)}s^{\asc_B(\pi)} u^{\pos_{a_n}(\pi)}=0
\end{eqnarray*}

Hence, there exists $g_{n-1}: \BB_{{\{a_1,a_2,\ldots,a_{n-2},a_{n} \}},
\pos_{a_n}(\pi) \neq n-1}^+ \rightarrow \BB_{{\{a_1,a_2,\ldots,a_{n-2},
a_{n} \}},\pos_{a_n}(\pi) \neq n-1}^-$ such that $g_{n-1}(\pi)$  
and $\pi$ has same number of descents.
Now we put `$a_{n-1}$' or `$\ob{a_{n-1}}$' in the last position of $\pi$ and $g_{n-1}(\pi)$ 
to get 
\begin{eqnarray*}
\sum_{\pi \in 
\BB^+_{{\{a_1,a_2,\ldots,a_{n} \}}},\pos_{a_{n-1}}(\pi) = n, \pos_{a_n } (\pi)\neq n-1 }
(-1)^{\inv_B(\pi)} t^{\des_B(\pi)}s^{\asc_B(\pi)} u^{\pos_{a_n}(\pi)}= \nonumber \\ 
\sum_{\pi \in \BB^-_{{\{a_1,a_2,\ldots,a_{n} \}}},\pos_{a_{n-1}}(\pi) = n, \pos_{a_n }  \neq n-1 }
(-1)^{\inv_B(\pi)}t^{\des_B(\pi)}s^{\asc_B(\pi)} u^{\pos_{a_n}(\pi)}	
\end{eqnarray*} 
The proof of Claim 1 is complete.

{\bf Claim 2:} For any $n$ positive integers $a_1,a_2,\ldots,a_n$, 
\begin{eqnarray}	
\label{eqn:an-1_not_at_last}
\sum_{\pi \in S_3} (-1)^{\inv_B(\pi)}t^{\des_B(\pi)}s^{\asc_B(\pi)} u^{\pos_{a_n}(\pi)}=0	
\end{eqnarray}

{\bf Proof of Claim 2:}
Consider $n-1$ positive integers, $a_1<a_2<\ldots<a_{n-2}<a_{n-1}$ and
by induction, we assume  that
\begin{eqnarray*}
\sum_{\pi \in \BB_{{\{a_1,a_2,\ldots,a_{n-2},a_{n-1} \}}},
\pos_{a_{n-1}}(\pi)  \neq n-1 }(-1)^{\inv_B(\pi)}
t^{\des_B(\pi)}s^{\asc_B(\pi)} u^{\pos_{a_{n-1}}(\pi)}=0
\end{eqnarray*}

Hence, there exists $f_{n-1}: 
\BB_{{\{a_1,a_2,\ldots,a_{n-1} \}}, \pos_{a_{n-1}}(\pi) \neq n-1}^+ 
\rightarrow 
\BB_{{\{a_1,a_2,\ldots,a_{n-1} \}},\pos_{a_{n-1}}(\pi) \neq n-1}^-$ such
that 
for all $\pi \in 
\BB_{{\{a_1,a_2,\ldots,a_{n-1} \}}, \pos_{a_{n-1}}(\pi) \neq n-1}^+$,
$\des_B(f_{n-1}(\pi)) = \des_B(\pi)$ and $ \pos_{a_{n-1}}(f_{n-1}(\pi))= \pos_{a_{n-1}}(\pi)
$.
So, we can put `$a_n$'or `$\ob{a_n}$' in the same position (except the last position) of $\pi$ and 
$f_{n-1}(\pi)$ to get 
\begin{eqnarray*}	
\label{eqn:an-1_not_at_last2}
\sum_{\pi \in \BB^+_{{\{a_1,a_2,\ldots,a_{n} \}}},
\pos_{a_{n-1}}(\pi)  \neq n, \pos_{a_{n} }(\pi) \neq n  }(-1)^{\inv_B(\pi)}t^{\des_B(\pi)}s^{\asc_B(\pi)}
u^{\pos_{a_n}(\pi)}= \nonumber \\ \sum_
{\pi \in \BB^-_{{\{a_1,a_2,\ldots,a_{n} \}}},
\pos(a_{n-1})(\pi)  \neq n, \pos_{a_{n}}(\pi)  \neq n  }(-1)^{\inv_B(\pi)}t^{\des_B(\pi)}s^{\asc_B(\pi)}
u^{\pos_{a_n}(\pi)}	\end{eqnarray*}
The proof of Claim 2 is complete.

We sum up the results for $S^1, S^2$ and $S^3$. Summing Lemma \ref{lem:half_pos_neg_sum},
\eqref{eqn:an-1_at_last_an_not_at_penultimate} and \eqref{eqn:an-1_not_at_last},  we get 
\begin{eqnarray}
\sum_{\pi \in \BB_{{\{a_1,a_2,\ldots,a_{n} \}}},
	\pos_{a_{n}}(\pi) \neq n  }(-1)^{\inv_B(\pi)}t^{\des_B(\pi)}s^{\asc_B(\pi)}
u^{\pos_{a_n}(\pi)}= 0.
\end{eqnarray} 
The proof is complete. 
\end{proof}

In Lemma \ref{lem:signed_sum_zero}, we set $u=1$ and $a_i=i$ for $1 \leq i \leq n$ and get 
the following corollary. 
\begin{corollary}
\label{des_over_Bnplus-equals_half_Des_over_Bn} 
For any positive integer $n$, the following is true: 
\begin{eqnarray}  
\label{eqn:signed_des_asc_over_Bn_equals_0}
\sum_{\pi \in \BB_{n}^+, \np(\pi) \neq n }
t^{\des_B(\pi)}s^{\asc_B(\pi)}=\sum_{\pi \in \BB_{n}^-, \np(\pi) \neq n }
t^{\des_B(\pi)}s^{\asc_B(\pi)}=
\frac{1}{2}
\sum_{\pi \in \BB_{n}, \np(\pi) \neq n }
t^{\des_B(\pi)}s^{\asc_B(\pi)}
\end{eqnarray}
\end{corollary}

\subsection{Main results for $\BB_n$}
\begin{theorem}
\label{thm:type_B_recurrence}
Let $n \geq 2$ be a positive integer.  The polynomials $B_n^+(s,t)$ and 
$B_n^-(s,t)$ satisfy the following recurrence. 
\begin{eqnarray}
\label{rec_for_Bn+}
B_n^+(s,t) & = & sB_{n-1}^+(s,t)+tB_{n-1}^-(s,t)+stDB_{n-1}(s,t), \\
\label{rec_for_Bn-}
B_n^-(s,t) & = & sB_{n-1}^-(s,t)+tB_{n-1}^+(s,t)+stDB_{n-1}(s,t).
\end{eqnarray}
\end{theorem}
	
\begin{proof}
We prove \eqref{rec_for_Bn+} first. Recall 
$B_n^+(s,t)=\sum_{\pi \in \BB_n^+} t^{\des_B(\pi)} s^{\asc_B(\pi)} $. 
Consider the contribution to the right hand side from $\pi \in \BB_n+$ with 
`$n$' or `$\ob{n}$' occurring in position $k$ for all possible choices of $k$. 
It is easy to see that 
\begin{eqnarray*}
B_n^+(s,t)=sB_{n-1}^+(s,t)+tB_{n-1}^-(s,t)+
\sum_{\pi \in \BB_{n}^+, \np(\pi) \neq n }
t^{\des_B(\pi)}s^{\asc_B(\pi)}.
\end{eqnarray*} 
Here, $sB_{n-1}^+(s,t)$ accounts for the contribution of all $\pi \in \BB_n^+$ 
in which the letter `$n$' appears in the last position and 
$tB_{n-1}^-(s,t)$ is the contribution of all $\pi \in \BB_n^+$ in 
which the letter `$\ob{n}$' appears in the last position.
Further, from corollary \ref{des_over_Bnplus-equals_half_Des_over_Bn} we get
\begin{eqnarray} \label{eqn: recurrence bivariate B_n +}
B_n^+(s,t) & = & sB_{n-1}^+(s,t)+tB_{n-1}^-(s,t)+
    \frac{1}{2}\sum_{\pi \in \BB_{n}, \np(\pi) 
    	\neq n }t^{\des_B(\pi)}s^{\asc_B(\pi)}  \nonumber \\ \nonumber
    & = & sB_{n-1}^+(s,t)+tB_{n-1}^-(s,t)+stDB_{n-1}(s,t)
\end{eqnarray}
 The last line follows from the fact that $2stDB_{n-1}(s,t)$ is the contribution of all the 
permutations in $\BB_{n}$ where `$n$' or `$\ob{n}$' is not in the final position,
completing the proof of
 \eqref{rec_for_Bn+}.
The proof of \eqref{rec_for_Bn-} is identical and hence omitted.
\end{proof}

The following is the main result of this Section.
	
\begin{theorem}
\label{thm:main_result_for_Bn}
For all positive integers $n \geq 2$, both $B_n^+(s,t)$ 
and $B_n^-(s,t)$ are gamma-positive with the same center of 
symmetry $\frac{n}{2}$ iff $n \equiv 0$ (mod 2). 
\end{theorem}
\begin{proof} 
By Lemma \ref{lem:0_mod_2 for B_n^+}, we clearly need $n \equiv 0$ (mod 2).
We use inductionon $n$.  The base case is when $n=2$.  We have 
$B_2^-(s,t)=4st$ and $B_2^+(s,t)=s^2+2st+t^2$, both of which are gamma 
positive with center of symmetry $1$.  By induction, assume 
that for $n-2=2m-2$, both $B_{n-2}^-(s,t)$	and 
$B_{n-2}^+(s,t)$ are gamma positive with center of symmetry $m-1$. 
By applying the recurrence relations \eqref{rec_for_Bn+} and 
\eqref{rec_for_Bn-} twice, we get both
\begin{eqnarray} 
B_n^+(s,t) & = &(s^2+2st+t^2)B_{n-2}^+(s,t)+ 
	4st B_{n-2}^-(s,t)  \nonumber\\
&  & + 4st(s+t)D B_{n-2}(s,t)+2s^2t^2D ^2 B_{n-2}(s,t) 
\label{mainrecBtypeplus} \\
B_n^-(s,t) & = & (s^2+2st+t^2)B_{n-2}^-(s,t)+  4st  B_{n-2}^+(s,t)  \nonumber
\\
& & + 4st(s+t)D B_{n-2}(s,t)+2s^2t^2D^2B_{n-2}(s,t) 
\label{mainrecBtypeminus} 
\end{eqnarray}


Each of the four individual terms in \eqref{mainrecBtypeplus} and
\eqref{mainrecBtypeminus} clearly has center of symmetry $m$.  
Thus, both $B_{2m}^+(s,t)$ and $B_{2m}^-(s,t)$ are gamma positive with 
center of symmetry $m$, completing the proof.
\end{proof}
	
\begin{theorem}
\label{sum_of_2_for_odd_B_type}
For all positive odd integers $n=2m+1$, $n \geq 3$,
$B_{n}^-(t)$ and $B_{n}^+(t)$ are the sum of two gamma positive
polynomials.
\end{theorem}
\begin{proof}
  $B_{n}^+(s,t)$ and $B_{n}^-(s,t)$ satisfy the recurrences 
\eqref{rec_for_Bn+} and \eqref{rec_for_Bn-} which  are identical to 
the recurrences \eqref{eqn:n_even_pi_even_recurrence} and 
\eqref{eqn:n_even_pi_odd_recurrence} which we used to prove
Theorem \ref{thm:An_more_gamma_positive}.  Hence, the proof of this
result follows in an identical manner to the proof of Theorem 
\ref{thm:An_more_gamma_positive}.
\end{proof}

Similar to Corollary \ref{cor:tanimoto_rec_2m}, 
we get the following recurrences for the numbers 
$b_{n,k}^+$ and $b_{n,k}^-$ using Theorem \ref{thm:type_B_recurrence}. 
\begin{corollary} For all positive integers $n$ and $0\leq k \leq n$, $b_{n,k}^+$ and $b_{n,k}^-$ satisfy the following recurrence relations.
\begin{enumerate}
	\item $b_{n,k}^+= 2kb_{n-1,k}^{-} + (2n-2k+1)b_{n-1,k-1}^{-} + b_{n-1,k}^{+}, $
	\item $b_{n,k}^-= 2kb_{n-1,k}^{+} + (2n-2k+1)b_{n-1,k-1}^{+} + b_{n-1,k}^{-}.  $
\end{enumerate}
\end{corollary}

\section{Type-D Coxeter groups}
Recall that $\DD_n \subseteq \BB_n$ is the subset of type-B 
permutations that have an even number of negative signs. 
Let $w = w_1, w_2,\ldots, w_n \in \DD_n$.
The following combinatorial definition of type-D inversions
is well known, see for example Petersen's book \cite[Page 302]{petersen-eulerian-nos-book}
$\inv_D(w) = \inv_A(w) + | \{1 \leq i < j \leq n:  -w_i > w_j \}|$.
Here $\inv_A(w)$ is computed with respect to the usual order
on $\ZZ$.  Let $\DD_n^+ = \{w \in \DD_n: \inv_D(w) \> \mbox{ is even} \}$  
and let $\DD_n^- = \DD_n - \DD_n^+$.  

We next give the combinatorial definition of type-D descents. 
Let $w= w_1,w_2, \ldots,w_n \in \DD_n$.  Then 
$\des_D(w) = \des(w) + \chi(w_1 + w_2 < 0)$ where $\chi(Z)$ is a
0/1 indicator function taking value 1 if $Z$ is true
and taking value 0 otherwise.  Here, $\des(w)$ is as defined in the 
type-A case.  See Petersen's book 
\cite[Page 302]{petersen-eulerian-nos-book}
for this definition.  Let $\asc_D(w) = n - \des_D(w)$ and define

\vspace{-6 mm}

\begin{eqnarray}
\label{eqn:Dn} 
D_n(t) & = & \sum_{w \in \DD_n} t^{\des_D(w)} 
\mbox{ 
\hspace{ 8 mm}
and 
\hspace{ 8 mm}
}
D_n(s,t)  =  \sum_{w \in \DD_n} t^{\des_D(w)} s^{\asc_D(w)}\\  
\label{eqn:DAn} 
D_n^+(t) & = & \sum_{w \in \DD_n^+} t^{\des_D(w)}  
\mbox{ 
\hspace{ 8 mm}
and 
\hspace{ 8 mm}
}
D_n^+(s,t)  =  \sum_{w \in \DD_n^+} t^{\des_D(w)} s^{\asc_D(w)}\\  
\label{eqn:DnAn} 
D_n^-(t) & = & \sum_{w \in \DD_n^-} t^{\des_D(w)} 
\mbox{ 
\hspace{ 8 mm}
and 
\hspace{ 8 mm}
}
D_n^-(s,t)  =  \sum_{\pi \in \DD_n^-} t^{\des_D(\pi)} s^{\asc_D(\pi)}
\end{eqnarray}

The polynomials $D_n(t)$ are called the type-D Eulerian polynomials.
Stembridge \cite{stembridge-coxeter-cones} showed that the polynomials 
$D_n(t)$ are gamma positive for all $n \geq 1$ (see Petersen 
\cite[Section 13.4]{petersen-eulerian-nos-book}).  
For $D_n(s,t)$ as defined \eqref{eqn:Dn}, we are unable to get recurrences 
similar to \eqref{eqn:rec_biv_eulerian} or \eqref{eqn:rec_biv_eulerian_B}.
Our proof of the gamma positivity of $D_n^+(t)$ follows the proof that
type-D Eulerian polynomials are gamma positive 
from Petersen's book \cite[Section 13.4]{petersen-eulerian-nos-book}.

Since $\DD_n \subseteq \BB_n$, each $w \in \BB_n$ has both the 
type-B and the type-D inversion numbers (though the type-D inversion
number does not have Coxeter theoretic meaning).  Define 
$\BB_{n,\DD}^+ = \{ w \in \BB_n: \inv_D(w) \> \mbox{ is even} \}$ 
and  $\BB_{n,\DD}^- = \{ w \in \BB_n: \inv_D(w) \> \mbox{ is odd} \}$.  

\begin{lemma}
\label{lemma:half_Dn_plusminus}
With the above notation, we have
\begin{eqnarray}
\label{eqn:half_Dn_plus}
\sum_{w \in \DD_n^+} t^{\des_D(w)}s^{\asc_D(w)} = \sum_{w  \in 
	(\BB_{n,\DD}^+ - \DD_n^+)} t^{\des_D(w) }s^{\asc_D(w)} 
= \frac{1}{2} \sum_{w \in \BB_{n,\DD}^+}t^{\des_D(w)}s^{\asc_D(w)} \\
\label{eqn:half_Dn_minus}
\sum_{w \in \DD_n^-} t^{\des_D(w)}s^{\asc_D(w)}  = \sum_{w  \in 
	(\BB_{n,\DD}^- - \DD_n^-)} t^{\des_D(w) }s^{\asc_D(w)} 
= \frac{1}{2} \sum_{w \in \BB_{n,\DD}^-}t^{\des_D(w)}s^{\asc_D(w)} 
\end{eqnarray}
\end{lemma}

\begin{proof} 
We prove \eqref{eqn:half_Dn_plus} first. For this, 
consider the bijection from $\DD_n^+$
to $\BB_{n,\DD}^+ - \DD_n^+$ by flipping the first
 element i.e consider the map 
$f:\DD_n^+ \mapsto \BB_{n,\DD}^+ - \DD_n^+$ given by 
$f(u_1,u_2,\ldots,u_n)=\ob{u_1},u_2,\ldots,u_n$.  
Clearly  $f^2=id$  and $f$ preserves the 
number of descents and parity of $\inv_{D}u$. 
The proof of \eqref{eqn:half_Dn_minus} is similar and hence omitted. 
\end{proof}

Let $u \in \SSS_n$ and let $\beta(u)$ denote the set of $2^n$ elements 
of $\BB_n$ obtained by adding negative signs to the entries of $w$ in
all possible ways.  

\begin{lemma}
\label{lemma:all_in_same_parity}
Let $u \in \SSS_n$.  Then, for all $w \in  \beta(u)$, 
either $w \in \BB_{n,\DD}^+$ or $w \in \BB_{n,\DD}^-$.
Hence, for any $u \in \SSS_n$, $u \in \BB_{n,\DD}^+$
if and only if $u \in \AAA_n$. 
\end{lemma}
\begin{proof}
The point to prove is that all the $2^n$ elements of 
$\beta(u)$ have the same
parity of the number of type-D inversions. For any element 
$w=w_1,w_2,\ldots,w_n \in \beta(u)$, 
let $w_{\bar{l}} = w_1,w_2,\ldots,-w_{l},\ldots,w_n$ 
denote the permutation obtained
by flipping the sign of the $l$-th entry of $w$.  
Note that we only need to show that for any $l$ $w$ and 
$w_{\bar{l}}$ have the same parity of number of type-D inversions. 
Note that 

%


$\inv_D(w)- \inv_D(w_{\bar{l}}) = |\{1 \leq i < j=l :w_i>w_j\}|+
|\{1 \leq i < j=l :-w_i>w_j\}| -
\{|1 \leq i < j=l :w_i>-w_j\}|- 
|\{1 \leq i < j=l :-w_i>-w_j\}|$

We show that $|\{1 \leq i < j=l :w_i>w_j\}|+
|\{1 \leq i < j=l :-w_i>w_j\}|$ and
$\{|1 \leq i < j=l :w_i>-w_j\}|+ 
|\{1 \leq i < j=l :-w_i>-w_j\}|$ have the same parity. 
For this,  let $i<l$.  We show that $\chi (w_i>w_l)
+ \chi (-w_i>w_l)$ and $\chi (w_i>-w_l) + \chi (-w_i>-w_l)$ 
have the same parity.  This follows from the simple observation 
below which we divide in cases: 

case i: If $w_i >w_l$ and $-w_i \leq w_l$, then 
$-w_i \leq -w_l$ and $w_i >-w_l$.  

case ii: If $w_i >w_l$ and $-w_i > w_l$, then 
$-w_i \leq -w_l$ and $w_i \leq -w_l$.  

case iii: If $w_i \leq w_l$ and $-w_i \leq w_l$, then 
$-w_i > -w_l$ and $w_i >-w_l$.  

case iv: If $w_i \leq w_l$ and $-w_i > w_l$, then 
$-w_i > -w_l$ and $w_i \leq -w_l$. 

Hence, $w$ and $w_{\bar{l}}$ have same parity of Type-D 
inversions and the proof is complete.  
\end{proof}

By Lemma \ref{lemma:all_in_same_parity}, we get
\begin{eqnarray}\label{eqn:all_in_same_parity_plus}
\sum_{w \in \BB_{n,\DD}^+}t^{\des_D(w)}s^{\asc_D(w)}  =  
\sum_{u \in \SSS_n \cap \BB_{n,\DD}^+}\sum_{w \in \beta(u)}t^{\des_D(w)} s^{\asc_D(w)} 
= \sum_{u \in \AAA_n}\sum_{w \in \beta(u)}t^{\des_D(w)} \\
\label{eqn:all_in_same_parity_minus}
\sum_{w \in \BB_{n,\DD}^-}t^{\des_D(w)}s^{\asc_D(w)}  =  
\sum_{u \in \SSS_n \cap \BB_{n,\DD}^-}\sum_{w \in \beta(u)}t^{\des_D(w)} 
=  \sum_{u \in \SSS_n-\AAA_n} \sum_{w \in \beta(u)}t^{\des_D(w)}
s^{\asc_D(w)}  
\end{eqnarray}

The proof of gamma positivity of $D_n(t)$ given in Petersen's book 
\cite[Page 305]{petersen-eulerian-nos-book} shows that it is possible
to assign values $c_i(u)$ for $1 \leq i \leq n$ to $u \in \SSS_n$ such that 
$\sum_{w \in \beta(u)}t^{\des_{D}w} = c_1(u)c_2(u)\ldots c_n(u).$
We work with the following bivariate version.

\begin{equation}
c_1(u)c_2(u)=\begin{cases}
2s^2+2t^2 & \text {if $u_1<u_2<u_3$}.\\
2st(s+t) & \text {if $u_1<u_2>u_3$}.\\
4st & \text {if $u_1>u_2<u_3$}. \\
2st(s+t) & \text {if $u_1>u_2>u_3$}. 
\end{cases}
\end{equation}
and for all $j \geq 3$ 
\begin{equation}
c_j(u)=\begin{cases}
2t & \text {if $u_{j-1}<u_j>u_{j+1}$}.\\
2s & \text {if $u_{j-1}>u_j<u_{j+1}$}.\\
s+t & \text{ otherwise  }
\end{cases}
\end{equation}

For a permutation $u=u_1,u_2,u_3,\ldots ,u_n \in \SSS_n$ with 
$u_1<u_2<u_3$, let $u'=u_3,u_1,u_2,\ldots,u_n$. It is easy to see that 
 $u \in \BB_{n,\DD}^+$ if and only if $u' \in \BB_{n,\DD}^+$.  We partition  
 $\AAA_n$ as union of the following five  disjoint sets 
 $\AAA_n^1=\{u=u_1,u_2,u_3,\ldots,u_n \in \AAA_n: u_1<u_2>u_3 \}$ , $\AAA_n^2=\{u=u_1,u_2,u_3,\ldots,u_n \in \AAA_n: u_1>u_2>u_3 \}$,  $\AAA_n^3=\{u=u_1,u_2,u_3,\ldots,u_n \in \AAA_n: u_1<u_2<u_3\}$, 
 $\AAA_n^4=\{u=u_1,u_2,u_3,\ldots,u_n \in \AAA_n: u_1>u_2<u_3,u_1>u_3\}$,
 $\AAA_n^5=\{u=u_1,u_2,u_3,\ldots,u_n \in \AAA_n: u_1>u_2<u_3,u_1<u_3\}$. 

For a permutation $u \in \SSS_n$, define
 $\lpk(u)=|\{1\leq i <n : u_{i-1} <u_i >u_{i+1} \}|$ where $u_0=0$.

\begin{lemma}
\label{lemma:main_lemma_for_Dn}
For positive integers $n\geq 3$, $|\AAA_n^3| = |\AAA_n^4|$. 
Moreover, $
\displaystyle 
\sum_{u \in \AAA_n^3 \cup \AAA_n^4 }
\sum_{w \in \beta(u)}t^{\des_D(w)}s^{\asc_D(w)}\displaystyle$ 
is a gamma positive polynomial with all gamma coefficients being even.
\end{lemma}

\begin{proof}
Define $f:\AAA_n^3 \mapsto \AAA_n^4$ by
$f(u_1,u_2,u_3,u_4,\ldots,u_n)=u_3,u_1,u_2,u_4,\ldots,u_n$. Clearly,
 $\pi \in \AAA_n$ iff
$f(\pi) \in \AAA_n$.  Further, $f$ is a bijection as its inverse is 
clearly  $g:\AAA_n^4 \mapsto \AAA_n^3$ given by $g(u_1,u_2,u_3,u_4,\ldots,u_n)=u_2,u_3,u_1,u_4,\ldots,u_n$. 
Hence, we have 

\begin{eqnarray*}
\sum_{u \in \AAA_n^3 \cup \AAA_n^4 } \sum_{w \in \beta(u)}t^{\des_D(w)}s^{\asc_D(w)}\displaystyle 
& = & \sum_{u \in \AAA_n^3 }\displaystyle(\sum_{w \in \beta(u)}t^{\des_D(w)}s^{\asc_D(w)}\displaystyle)+\sum_{f(u) \in \AAA_n^4 }\displaystyle(\sum_{w \in \beta(f(u))}t^{\des_D(w)}s^{\asc_D(w)}\displaystyle)\\
&=&	\sum_{u \in \AAA_n^3 }(2s^2+2t^2)\prod_{i=3}^{n}c_i(u)+\sum_{u \in \AAA_n^3 }4st\prod_{i=3}^{n}c_i(u)\\
&=&	\sum_{u \in \AAA_n^3  }\displaystyle2(s+t)^2\prod_{i=3}^{n}c_i(u)\\
&=&	\sum_{u \in \AAA_n^3  }\displaystyle2(4st)^{\lpk(u)}(s+t)^{n-2\lpk(u)}\\
 \end{eqnarray*}

In the last line, we used bivariate version of   \cite[Eqn (13.10)]{petersen-eulerian-nos-book} 
from Petersen's book.  The proof is complete.
\end{proof}

\begin{theorem}
\label{thm:main_result_for_Dn}
For all positive integers $n \geq 3$, the polynomials $D_n^+(s,t)$ and $D_n^-(s,t)$  
are gamma non-negative. 
\end{theorem}

\begin{proof}
We first prove that the polynomials $D_n^+(s,t)$  are gamma positive.  
From Lemma \ref{lemma:half_Dn_plusminus}, we get 
\begin{eqnarray*}
\label{eqn:half_Dn_pluss}
\sum_{w \in \DD_n^+} t^{\des_D(w)}s^{\asc_D(w)} = \sum_{w  \in (\BB_{n,\DD}^+ 
	- \DD_n^+)} t^{\des_D(w) }s^{\asc_D(w)}
= \frac{1}{2} \sum_{w \in \BB_{n,\DD}^+}t^{\des_D(w)}s^{\asc_D(w)}
\end{eqnarray*}

By \eqref{eqn:all_in_same_parity_plus}, we get 
\begin{eqnarray}
\sum_{w \in \DD_n^+} t^{\des_D(w)}s^{\asc_D(w)}&=&
\frac{1}{2} \sum_{u \in \AAA_n}(\sum_{w \in \beta(u)}t^{\des_D(w)})s^{\asc_D(w)} 
= \frac{1}{2}\sum_{i=1}^{5}\left(\sum_{u \in \AAA_n^{i}}\sum_{w \in \beta(u)}t^{\des_D(w)}s^{\asc_D(w)}\right)\nonumber\\
&=& \frac{1}{2}\sum_{i=1,2,5}\left(\sum_{u \in \AAA_n^{i}}\sum_{w \in \beta(u)}t^{\des_D(w)}s^{\asc_D(w)}\right) \nonumber \\
& & 
\label{eqn:breaking_into_disjoint_sets}
+\frac{1}{2}\sum_{i=3,4}\left(\sum_{u \in \AAA_n^{i}}\sum_{w \in \beta(u)}t^{\des_D(w)}s^{\asc_D(w)}\right)
\end{eqnarray}
  		
Clearly, if $u_1<u_2>u_3$, or $u_1>u_2>u_3$ or $u_3>u_1>u_2$, 
then $\sum_{w \in \beta(u)}t^{\des_D(w)}s^{\asc_D(w)}=(4st)^{\lpk(u)}(s+t)^{n-2\lpk(u)}$.  
So, the distribution of $\des_D(w) $ over these in each of the 
three cases i.e over  $\AAA_n^1$,$\AAA_n^2$ 
and $\AAA_n^5$ is gamma positive and further, $u$ in any of these 
three sets has at least one left peak. 
Thus, the gamma-coefficients are even.  Hence, the first term in \eqref{eqn:breaking_into_disjoint_sets} is 
gamma positive. The gamma positivity of the second term follows from 
Lemma \ref{lemma:main_lemma_for_Dn}. 
The proof for $\DD_n^-(s,t)$ follows identical steps and hence is omitted.  
The proof
is complete.
\end{proof}

\section{Questions and Conjectures}
In this section, we raise a few questions and make a few conjectures.  
The most important question is to find an interpretation for
gamma coefficients.

\begin{question}
Can we find an interpretation for the gamma coefficients that occur 
in Theorems \ref{thm:An_more_gamma_positive}, 
\ref{thm:An_still_more_gamma_positive}, 
\ref{thm:main_result_for_Bn}, \ref{sum_of_2_for_odd_B_type} and
\ref{thm:main_result_for_Dn}?
\end{question}

For $\pi \in \SSS_n$, denote $\ides(\pi) = \des(\pi^{-1})$ and define the two-sided
Eulerian polynomial by 
$\ATS_n(s,t) = \sum_{\pi \in \SSS_n} t^{\des(\pi)+1}s^{\ides(\pi)+1} = \sum_{i,j} a_{n,i,j}t^i s^j$.  
The polynomial $\ATS_n(s,t)$ is known to satisfy $a_{n,i,j}  = a_{n,j,i}$ and
$a_{n,i,j} = a_{n,n-i,n-j}$.  Gessel 
(see \cite[Conjecture 10.2]{branden-actions_on_perms_unimodality_descents}) conjectured that 
$\ATS_n(s,t) = \sum_{i,j} \gamma_{n,i,j} (st)^i (s+t)^j(1+st)^{n+1-2i-j}$ with 
$\gamma_{n,i,j} \geq 0$.  This was recently proved by Zhicong Lin 
\cite{zhicong-proof-gessel-gamma-positive}.  See the paper by Adin et~al 
\cite{adin-bagno-etal-two-sided-gamma-positivity} as well.

When $n \equiv 0,1$ (mod 4), define 
$\ATS_n^+(s,t) = \sum_{\pi \in \AAA_n} t^{\des(\pi)+1}s^{\ides(\pi)+1} = \sum_{i,j} a_{n,i,j}^+t^i s^j$ and
likewise define
$\ATS_n^-(s,t) = \sum_{\pi \in \SSS_n - \AAA_n} t^{\des(\pi)+1}s^{\ides(\pi)+1} = \sum_{i,j} a_{n,i,j}^-t^i s^j$.

Similarly, when we sum over $\BB_n^+$ and $\BB_n^-$, define 
$\TSB_n^+(s,t) = \sum_{\pi \in \BB_n^+} t^{\des_B(\pi)}s^{\ides_B(\pi)} = \sum_{i,j} a_{n,i,j}^+t^i s^j$ and
$\TSB_n^-(s,t) = \sum_{\pi \in \BB_n^-} t^{\des_B(\pi)}s^{\ides_B(\pi)} = \sum_{i,j} a_{n,i,j}^+t^i s^j$.

Similarly, when we sum over $\DD_n^+$ and $\DD_n^-$, define 
$\TSD_n^+(s,t) = \sum_{\pi \in \DD_n^+} t^{\des_D(\pi)}s^{\ides_D(\pi)} = \sum_{i,j} a_{n,i,j}^+t^i s^j$ and
$\TSD_n^-(s,t) = \sum_{\pi \in \DD_n^-} t^{\des_D(\pi)}s^{\ides_D(\pi)} = \sum_{i,j} a_{n,i,j}^+t^i s^j$.

\begin{conjecture}

\begin{enumerate}
\item When $n \equiv 0,1$ (mod 4), both polynomials $\ATS_n^+(s,t)$ and $\ATS_n^-(s,t)$ can be written as
$\ATS_n^+(s,t) = \sum_{i,j} \gamma_{n,i,j}^+ (st)^i (s+t)^j(1+st)^{n+1-2i-j}$ and 
$\ATS_n^-(s,t) = \sum_{i,j} \gamma_{n,i,j}^- (st)^i (s+t)^j(1+st)^{n+1-2i-j}$ 
with all $\gamma_{n,i,j}^+, \gamma_{n,i,j}^- \geq 0$.  

\item When $n \equiv 0$ (mod 2), both the polynomials $\TSB_n^+(s,t)$ and 
$\TSB_n^-(s,t)$ can be written as
$\TSB_n^+(s,t) = \sum_{i,j} \gamma_{n,i,j}^{B,+} (st)^i (s+t)^j(1+st)^{n-2i-j}$ and 
$\TSB_n^-(s,t) = \sum_{i,j} \gamma_{n,i,j}^{B,-} (st)^i (s+t)^j(1+st)^{n-2i-j}$ 
with all $\gamma_{n,i,j}^{B,+}, \gamma_{n,i,j}^{B,-} \geq 0$.  

\item For all positive integers $n$, both $\TSD_n^+(s,t)$ and $\TSD_n^-(s,t)$ can be 
written as 
$\TSD_n^+(s,t) = \sum_{i,j} \gamma_{n,i,j}^{D,+} (st)^i (s+t)^j(1+st)^{n-2i-j}$ and 
$\TSD_n^-(s,t) = \sum_{i,j} \gamma_{n,i,j}^{D,-} (st)^i (s+t)^j(1+st)^{n-2i-j}$ 
with all $\gamma_{n,i,j}^{D,+}, \gamma_{n,i,j}^{D,-} \geq 0$.  
\end{enumerate}

\end{conjecture}

It is well known that if a univariate polynomial $f(t)$ is palindromic, has positive coefficients
and is real rooted, then it is gamma positive (see \cite[Chapter 4]{petersen-eulerian-nos-book}).  
Several of the polynomials in this paper seem to be real rooted.  

\begin{conjecture}
\label{conj:real_rooted}
When $n \geq 4$ with $n \equiv 0,1$ (mod 4), the polynomials $A_n^+(t)$ and $A_n^-(t)$ are real rooted.  
When $n \geq 2$, with $n \equiv 0$ (mod 2), the polynomials $B_n^+(t)$ and $B_n^-(t)$ are real rooted.  
When $n \geq 2$, the polynomials $D_n^+(t)$ and $D_n^-(t)$ are real rooted.  
\end{conjecture}

\section*{Acknowledgement}
The first author acknowledges support from a CSIR-SPM
Fellowship.
The second author acknowledges support from project grant 
P07 IR052, given by IRCC, IIT Bombay.

\bibliographystyle{acm}
\bibliography{main}
\end{document}